\documentclass[12pt,reqno]{amsart}
    \usepackage{amssymb,amsmath,amsthm,newlfont}
    \usepackage{psfrag}
    \usepackage[dvips,final]{graphicx}
\usepackage[T1]{fontenc}
\usepackage{color}
\numberwithin{equation}{section}
\newcommand{\T}{\mathbb{T}}

\newcommand{\N}{\mathbb{N}}
\newcommand{\R}{\mathbb{R}}
\newcommand{\C}{\mathbb{C}}

\newcommand{\D}{\mathbb{D}}
\newcommand{\K}{\mathbb{K}}
\newcommand{\LL}{\mathbb{L}}
\newcommand{\M}{\mathbb{M}}

\newcommand{\cI}{{\mathcal{I}}}
\newcommand{\cD}{{\mathcal{D}}}

\newcommand{\cH}{{\mathcal{H}}}

\newcommand{\beqa}{\begin{eqnarray*}}
\newcommand{\bea}{\begin{eqnarray}}
\newcommand{\eeqa}{\end{eqnarray*}}
\newcommand{\eea}{\end{eqnarray}}

\newcommand{\dst}{\displaystyle}

\newtheorem{thm}{\bf Theorem}[section]
\newtheorem{coro}[thm]{\bf Corollary}

\newtheorem*{prop*}{\bf Proposition}

\newtheorem{lem}[thm]{\bf Lemma}

\begin{document}

\title{Carleson's formula for some weighted Dirichlet spaces}

\date{\today}
\subjclass[2000]{primary 46E20; secondary 30C85, 30J99} 
\keywords{Dirichlet space, Carleson formula, Douglas formula}
\thanks{This work was supported by ANR FRAB: ANR-09-BLAN-0058-02}

\author[B. Bouya, A. Hartmann]{B. Bouya and A. Hartmann}

\address{Universit\'e de Bordeaux, CNRS, Bordeaux INP, IMB, UMR 5251, 351 Cours de la Lib\'eration, F-33400, Talence, France.}
\email{Andreas.Hartmann@math.u-bordeaux.fr}

\begin{abstract}
{We extend Carleson's formula to radially polynomially weighted Dirichlet spaces.}
\end{abstract}
\maketitle

{\it Dedication: This paper is dedicated to the memory of Mohamed Zarrabi who sadly past
away in december 2021. He was a very esteemed colleague which we all miss in Bordeaux.
The results presented here had mainly been elaborated a very long time 
ago when 
Brahim Bouya was a postdoc at the University Bordeaux 1. This paper is also the occasion to bring 
back memories of Brahim who left this world prematurely in 2020. Even though they did not
work explicitely together on Dirichlet spaces --- one of Mohamed's research directions --- 
Brahim was one of Mohamed's co-authors. It appears natural to present this work in this
special edition of the Moroccan Journal of Pure and Applied Analysis.}

\section{ \bf Introduction and statement of the main results.}
Let $\cD$ be the standard Dirichlet space of analytic functions with
square area integrable modulus of the derivative on the unit disk $\D$ of the complex plane $\C$
(see precise definitions below).
It is well known that $\cD$ is contained in the Hardy space $\cH^2$
(see for instance \cite{R}),
and thus that every function $f\in \cD$ has non-tangential boundary
values almost everywhere on $\T:=\partial\D$ which have square integrable modulus
on $\T$. Even more is true, those functions $f$ admit actually non-tangential
limits quasi-everywhere on $\T,$ see \cite{Beu, FKR, R}.

While the norm of a function $f\in\cD$ is {\it a priori} defined {\it via} the values
of its derivative on the unit disk $\D$ it is possible to express it
by its values on $\T$ {\it only}. Indeed, Douglas' formula (see \eqref{DouglasFormula})
gives a
characterization involving difference quotients on the boundary \cite{JD}.
We refer for instance to the survey paper \cite{R} and the textbook 
\cite{FKR} for more information on Dirichlet spaces.

A special attention in this connection was attracted by outer functions in $\cD$
since they are completely determined by their {\it moduli} on the boundary.
Indeed, a famous result by Carleson \cite{Car} states that
the norm of an outer function $f$ in $\cD$ can be completely recovered
from its moduli on the boundary (see \eqref{carlesonformula}).
Later, analogs
of Carleson's formula were established in other classes of analytic functions, such
as that given by Vinogradov and Shirokov \cite{VS} for the space of analytic functions with derivative in the classical Hardy space $\cH^p$ and also in \cite[Theorem 3.1]{Shi2} for some spaces of analytic functions smooth up to the boundary.

Another result that is worth being mentioned here is by Aleman \cite{Ale}
who characterizes the norm of some Dirichlet type functions
in terms of their moduli and involving mean oscillation of the
function's modulus with respect to harmonic measure, see also
\cite{Boe, Dya, Shi1} and the survey paper \cite{ARSW}, but this characterization uses also the values
of the modulus of $f$ inside the disk.
The aim of this paper is to generalize Carleson's result to weighted Dirichlet
spaces for which an analog of Douglas' formula is actually known (see
\eqref{DouglasFormulaalpha}). Without entering into the very definitions
of weighted Dirichlet spaces $D_{\mu}$ associated to a measure
$\mu$, we mention that when $\mu$ is supported on $\T$, Richter \cite{Ri}
introduced and studied these spaces
as part of his analysis of two-isometric operators. In \cite{RS} Richter and Sundberg give a Carleson type formula for the spaces $D_{\mu},$ when $\mu$ is supported on $\T.$ 

In this paper, we are interested in the case of polynomial radial weights in the disk.
In this situation, our characterization recovers Carleson's result in the
limiting situation when the weight becomes constant (with non optimal
constants however).
\\

In order to be more precise, we now introduce
the weighted Dirichlet spaces we are interested in.
Let $\cD_{\alpha}$ be the space of analytic functions $f$ on  $\D$ with a
finite  weighted Dirichlet integral
\begin{equation}\label{domega}
\cD_{\alpha}(f):=\frac{1}{\pi} \int_{\D}|f'(z)|^2(1-|z|)^\alpha dA(z), 
\end{equation}
where $A$ is the standard area Lebesgue measure and $0\leq\alpha<1$
is a real number. Equipped with the norm
\begin{equation}\label{nrm}
\|f\|^{2}_{\cD_{\alpha}}:= |f(0)|^{2}+\cD_{\alpha}(f),
\end{equation}
the space $\cD_{\alpha}$ becomes a Hilbert space.
The limit case $\cD:=\cD_{0}$ is the classical Dirichlet space, and the case $\alpha=1$ corresponds
to the classical Hardy space $\mathcal{H}^2$. We denote by
 $\mathcal{L}^{2}(\T)$ the space of complex valued functions with square integrable modulus on $\T$.
Note that we can define an equivalent norm in $\cD_{\alpha}$
by $(\|f\|_2^2+\cD_{\alpha}(f))^{1/2}$, where 
$\|f\|_2$ is the standard norm in $\mathcal{L}^2(\T)$.

In all what follows we suppose that
$h\in\mathcal{L}^{2}(\T)$  is a non negative function such that
\begin{eqnarray}\label{logint}
\int_{-\pi}^{\pi}\log h(t) dt>-\infty,
\end{eqnarray}
where we identify the circle and the real line $\R$ by
$h(t):=h(e^{it})$, $t\in\R$.
By well known Hardy space theory (see for instance \cite{Gar})
we can associate with $h$ the outer function $O_{h},$ defined by
$$O_{h}(z):=\exp\{u_{_h}(z)+iv_{_h}(z)\},\qquad z\in\D,$$
where $$u_{_h}(z):=\displaystyle\frac{1}{2\pi}\int_{-\pi}^{\pi}\text{Re}\Big(\frac{e^{i\varphi}+z}{e^{i\varphi}-z}\Big)\log h(\varphi)d\varphi,  \quad z\in\D,$$
and $v_{_h}$ is the harmonic conjugate of the harmonic function $u_{_h}$ given by
$$v_{_h}(z):=\frac{1}{2\pi}\int_{-\pi}^{\pi}\text{Im}\Big(\frac{e^{i\varphi}+z}{e^{i\varphi}-z}\Big)\log h(\varphi)d\varphi, \qquad z\in\D.$$
The non tangential limits of $|O_{h}|$ exist and coincide with $h$ on $\T$ almost everywhere with respect to Lebesgue measure.
When studying the Plateau problem, Jesse
Douglas \cite{JD} obtained the following formula
for $f\in \cH^2$,
\begin{equation}\label{DouglasFormula}
\cD(f)=\frac{1}{4\pi^2}\int_{-\pi}^{\pi}\int_{-\pi}^{\pi} \
\left| \frac{f(\theta)-f(\varphi) }{e^{i\varphi}-e^{i\theta}}\right|^2d\theta d\varphi,
\end{equation}
which expresses the Dirichlet integral in terms of values of $f$ on the boundary
$\T$ only. The formula generalizes to weighted spaces $\cD_{\alpha}$ where
equality is replaced by equivalence (see for instance \cite{DH, FKR}):
\begin{equation}\label{DouglasFormulaalpha}
\cD_{\alpha}(f)\asymp \int_{-\pi}^{\pi}\int_{-\pi}^{\pi} \
 \frac{|f(\theta)-f(\varphi)|^2 }{|e^{i\varphi}-e^{i\theta}|^{2-\alpha}}d\theta d\varphi.
\end{equation}
When $f$ is outer, then it is uniquely determined by the modulus of its boundary
values, and one may ask whether it is then possible to express the Dirichlet
integral by these moduli only.
In \cite{Car}, Carleson  proved the following formula
\begin{equation}\label{carlesonformula}
\cD(O_{h})=\frac{1}{4\pi^2}\int_{-\pi}^{\pi}\int_{-\pi}^{\pi} \frac{\big(h^2(\varphi)-h^2(\theta)\big)\log \frac{h(\varphi)}{h(\theta)}}{|e^{i\varphi}-e^{i\theta}|^2}d\varphi d\theta  ,
\end{equation}
which thus allows to express the
norm
of outer functions in $\cD$ by their moduli on the boundary. Carleson actually
proved a more general result taking into account also the inner part, but then, obviously, the Dirichlet integral is no longer
given by the modulus of its boundary values only, and one has to consider the zeros
of the Blaschke factor and the singular measure.
A main ingredient in the proof of
\eqref{carlesonformula} is the classical Stokes formula which is in fact
not adapted to the situation in $\cD_{\alpha}.$

A natural guess for a candidate replacing \eqref{carlesonformula} in
the space $\mathcal{D}_{\alpha}$ would be
$$\mathcal{C}_\alpha(h):=\int_{-\pi}^{\pi}\int_{-\pi}^{\pi} \frac{\big(h^2(\varphi)-h^2(\theta)\big)\log \frac{h(\varphi)}{h(\theta)}}{|e^{i\varphi}-e^{i\theta}|^{2-\alpha}}d\varphi d\theta.$$
However, as it turns out,
there are functions $f\in \mathcal{D}_{\alpha},$ when $0<\alpha<1$, for which
$\mathcal{C}_{\alpha}(|f|)$ is not finite, see Theorem \ref{example} below.
\\

Note that an elementary computation yields that
for strictly positive numbers $a$ and $b$ we have
\begin{eqnarray}\label{logequiv}
 0\le (a^2-b^2)\log\frac{a}{b}
 \asymp \left\{
 \begin{array}{ll}
 (a-b)^2, & \quad \text{ if } \frac{1}{2}b\leq a\leq2 b,\\
 a^2\log\frac{a}{b}, & \quad \text{ if }a\ge 2b,\\
 b^2\log\frac{b}{a}, & \quad  \text{ if }a\le \frac{1}{2}b.
 \end{array}
 \right.
\end{eqnarray}
So, in the characterization that we propose below, according to the three cases
appearing in \eqref{logequiv}, we will distinguish what happens on the different parts
of the circle when the quotient $h(\varphi)/h(\theta)$ is bigger than 2, less
than 1/2 or between 1/2 and 2.
In order to be more precise, we need to introduce some notation.
Let $\Lambda$ be the set of measurable functions
on $\T$ that are strictly positive a.e.\ with respect to Lebesgue measure.
For $h$ and $\lambda\in\Lambda,$ we set
\begin{equation}\label{harmnorm}
N_{\alpha}(h):=\int_{-\pi}^{\pi}\int_{-\pi}^{\pi}\frac{|h(\varphi)-h(\theta)|^2}{|e^{i\varphi}-e^{i\theta}|^{2-\alpha}}d\varphi d\theta,
\end{equation}
\begin{equation}\label{definnalpha}
n_{\alpha}(h,\lambda):=\int_{-\pi}^{\pi}h^{2}(\theta) \Big(\int_{h(\varphi)\leq\frac{1}{2}h(\theta)\atop |e^{i\varphi}-e^{i\theta}|\geq \lambda(\theta)}\frac{\log \frac{h(\theta)}{h(\varphi)}}{|e^{i\varphi}-e^{i\theta}|^2}d\varphi\Big)^{1-\alpha}d\theta
\end{equation}
and
\begin{equation}\label{defintildenalpha}
\widetilde{n}_{\alpha}(h,\lambda):=\int_{-\pi}^{\pi}h^{2}(\theta)\Big(\int_{h(\varphi)\leq\frac{1}{2}h(\theta)\atop |e^{i\varphi}-e^{i\theta}|\leq \lambda(\theta) } \frac{\log \frac{h(\theta)}{h(\varphi)}}{|e^{i\varphi}-e^{i\theta}|^{2-\alpha}}d\varphi\Big) d\theta.
\end{equation}
Observe that by the triangular inequality, we have
\begin{equation}\label{estimNalpha}
 N_{\alpha}(h)\lesssim \mathcal{D}_{\alpha}(O_h).
\end{equation}

For two real valued functions $k_1$ and $k_2$  and a positive constant $c$
we use the notation
$k_1\stackrel{c}{\asymp} k_2,$ to design $ c^{-1}k_2\leq k_1\leq c k_2.$
By $k_1 {\asymp} k_2$ and $k_1\lesssim k_2$ we mean respectively that there exists some non specified constant $c$
such that $k_1\stackrel{c}{\asymp} k_2$ and $k_1\leq c k_2$.

We are now in a position to state our  first main result.

\begin{thm}\label{ra}
Let $0\leq\alpha<1$ be a real number.
Let $h\in\mathcal{L}^{2}(\T)$ be a non negative function satisfying \eqref{logint}. Then
  \begin{equation}\label{theorem}
\|O_h\|^{2}_{\cD_{\alpha}}\stackrel{c_\alpha}{\asymp} \|h\|^{2}_{2}+ N_{\alpha}(h)+ \inf_{\lambda\in\Lambda}\{ n_{\alpha}(h,\lambda)+\widetilde{n}_{\alpha}(h,\lambda)\},
  \end{equation}
where $c_\alpha\asymp 1$ when
$\alpha\to 0$.
\end{thm}


Let us consider the special situation when $\alpha=0$.
It is clear that
formula \eqref{theorem} does not depend on the choice $\lambda\in\Lambda$
when $\alpha=0$. 
Hence, in this case, the theorem gives
an equivalent expression to Carleson's formula \eqref{carlesonformula}.

An immediate consequence of this result is the following observation.
\begin{coro}\label{cor1}A bounded outer function $O_h$
which is also bounded away from zero
is in $\cD_{\alpha}$ if and only if
\[
 N_{\alpha}(h)<\infty.
\]
\end{coro}
We include the simple proof of this fact here.
\begin{proof}
We have
\[
 c^{-1}\le h\le c,
\]
for some positive constant $c>1.$
Then, for almost all $\theta$ and $\varphi$,
\[
c^{-2}\le\frac{h(\theta)}{h(\varphi)}\le c^2.
\]
One could replace the constant $1/2$ appearing
in the definitions \eqref{definnalpha} and \eqref{defintildenalpha} by any other value in $(0,1),$ 
say $c^{-2}.$
In this case, the expressions $n_{\alpha}(h,\lambda),$ $\widetilde{n}_{\alpha}(h,\lambda)$ and $m(h,\lambda)$
are zero since we integrate over void domains.
\end{proof}

We shall now discuss an appropriate choice for the function
$\lambda$ in the above theorem. In order to do this we
associate  with $h$ and $\lambda\in\Lambda$
the following functions
\begin{equation}\label{ahlambda}
a_{_{h,\lambda}}(\theta):= \frac{1}{2\pi}\int_{h(\varphi)\leq\frac{1}{2}h(\theta)\atop |e^{i\varphi}-e^{i\theta}|\geq \lambda(\theta)}\frac{\log \frac{h(\theta)}{h(\varphi)}}{|e^{i\varphi}-e^{i\theta}|^2}d\varphi
\end{equation}
and
\begin{equation}\label{ahlambda2}
\widetilde{a}_{_{h,\lambda}}(\theta):=\frac{1}{2\pi}\int_{h(\varphi)\leq\frac{1}{2}h(\theta)\atop
|e^{i\varphi}-e^{i\theta}|\leq \lambda(\theta)}\log\frac{h(\theta)}{h(\varphi)}d\varphi.
\end{equation}
The function $\lambda\times a_{h,\lambda}$ has an interpretation as a Poisson integral 
at $z(\theta)=(1-\lambda(\theta))e^{i\theta}$ of the function
$\log\frac{h(\theta)}{h(\varphi)}\chi$ where $\chi$ is the characteristic function of
$\{\varphi:h(\varphi)\le \frac{1}{2}h(\theta),|e^{i\theta}-e^{i\varphi}|\ge\lambda\}$.
For this one can observe that $|e^{i\varphi}-e^{i\theta}|\asymp |e^{i\varphi}-z(\theta)|$.
Similarly,
$ \tilde{a}_{h,\lambda}/\lambda$ as a Poisson integral 
at $z(\theta)$ of the function
$\log\frac{h(\theta)}{h(\varphi)}\tilde{\chi}$ where $\tilde{\chi}$ is the characteristic function of
$\{\varphi:h(\varphi)\le \frac{1}{2}h(\theta),|e^{i\theta}-e^{i\varphi}|\le\lambda\}$.
Actually, $\lambda a_{h,\lambda}+\tilde{a}_{h,\lambda}/\lambda$ is equivalent to the 
Poisson integral of $\log \frac{h(\theta)}{h(\varphi)}\chi_{h(\varphi)\le h(\theta)/2}$ at $z(\theta)$.

We set
\begin{eqnarray}\label{nana1}
\mu_{_h}(\theta):=\sup\left\{\mu\in (0,1]:\sup_{0<\delta\le\mu}
\Big\{
\delta a_{_{h,\delta}}(\theta), 
\frac{\widetilde{a}_{_{h,\delta}}(\theta)}{\delta}\Big\}\leq 2\right\}.
\end{eqnarray}
This allows to state our second main result.

\begin{thm}\label{ra2}
Let $0\leq\alpha<1$ be a real number.
Let $h\in\mathcal{L}^{2}(\T)$ be a non negative function satisfying \eqref{logint} and such that $N_{\alpha}(h)<+\infty.$ Then
$\mu_{_h}\in\Lambda$ and
  \begin{equation}\label{theorem2}
  \|O_h\|^{2}_{\cD_{\alpha}}
\stackrel{c_{\alpha}}{\asymp}  \|h\|^{2}_{2}+ N_{\alpha}(h)+  n_{\alpha}(h)+\widetilde{n}_{\alpha}(h),
  \end{equation}
where $n_{\alpha}(h):=n_{\alpha}(h,\mu_{_h})$ and $\widetilde{n}_{\alpha}(h):=\widetilde{n}_{\alpha}(h,\mu_{_h}),$ 
and $c_{\alpha}\asymp 1$
when $\alpha\to 0$.
\end{thm}

It would be interesting to know whether 
$\cD_{\alpha}(O_h){\asymp} N_{\alpha}(h)+  n_{\alpha}(h)+\widetilde{n}_{\alpha}(h)$. Note
that both sides vanish for constant functions. 

We should make two more important observations here. First, though the condition of
Theorem \ref{ra2} might appear difficult to check at first glance, it confirms that as
in Carleson's result for $\mathcal{D}$,
the membership of an outer function $f$ in $\mathcal{D}_{\alpha}$ depends
on its modulus on $\T$ only, which seems to be of  interest in its own.

Second, as it turns out, there is a family of functions for which the quantities in \eqref{theorem2} can be
 estimated explicitely. As a result, for this family the quantities
$N_\alpha,$ $\cD_\alpha$ and  $\mathcal{C}_\alpha$ are shown to be not equivalent to each other.
This will be discussed in the last section where we consider the following class of functions
$h_{\beta}$:
\begin{equation}\label{p0}
 h_{\beta}(\theta):=
\left\{
  \begin{array}{ll}
\displaystyle \frac{1}{\theta^{\frac{\alpha}{2}}\log^\beta\frac{\gamma}{\theta}}, & \qquad \theta\in (0,\pi] , \\
 \displaystyle \frac{h_{\beta}(\pi)}{2},& \qquad \theta
 \in (-\pi,0),
  \end{array}
\right.
\end{equation}
where $\gamma=\pi e^{2\beta/\alpha}$ guaranteeing that $h_{\beta}$ is decreasing on
$(0,\pi)$.

Then we have the following result.

\begin{thm}\label{example}Let  $0<\alpha<1$ and $\beta>0$.
Then
\begin{enumerate}
  \item [1.] For $N_\alpha(h_{\beta})<+\infty$ it is necessary and sufficient that $\beta>\frac{1}{2}.$
  \item [2.] For $O_{h_{\beta}}\in\cD_\alpha$ it is necessary and sufficient that $\beta>1-\frac{1}{2}\alpha.$
  \item [3.] For $\mathcal{C}_\alpha(h_{\beta})<+\infty$  it is necessary and sufficient that $\beta>1.$
\end{enumerate}
\end{thm}

The paper is organized as follows. In the next section we present
some auxiliary results. Section \ref{muh} is devoted to presenting some properties related to the function $\mu_{_h}$ defined in \eqref{nana1}.
The proof of our main result being quite technical (though the main tools are rather elementary),
we have split it into two sections:
Section \ref{sect3} is devoted to the proof
of the sufficiency while the necessity is shown in Section \ref{sect4}.
In the last part of the paper we will prove Theorem \ref{example}.

\section{\bf Auxiliary results. \label{sect2}}

Let $f:=e^{u+iv}\in\cH^{2}(\D)$ be an outer function and let $0\leq\alpha<1$ be a real number. We define
$f_{r}$ to be the function
$$f_r(w):=f(rw),\qquad w\in\D,$$
where  $0\leq r<1$.
Clearly $f_r$ is holomorphic and thus continuous in a neighborhood of
the closed unit disk.
It is possible to check that (see e.g.\ \cite{FKR})
\begin{equation}\label{equiv}
\int_{0}^{1}(1-r)^{\alpha}r^m dr \asymp \frac{1}{(m+1)^{1+\alpha}},\qquad m\in\N,
\end{equation}
independently of $\alpha.$
By Parseval's identity and \eqref{equiv} we get
\begin{equation}\label{parseval}
\cD_{\alpha}(f) \asymp \sum_{n\geq1}|\hat{f}(n)|^{2}(1+n)^{1-\alpha}.
\end{equation}
In particular, when $\alpha=0$,
\begin{equation}\label{dfr}
\cD(f_r)\asymp \sum_{n\geq1}r^{2n}|\hat{f}(n)|^{2}(1+n), \qquad 0\leq r<1,
\end{equation}
which is actually an equality.
In all what follows we suppose that $0<\alpha<1.$
Using \eqref{equiv} and \eqref{dfr}
\begin{equation}\label{p2}
\alpha\int_{0}^{1}\cD(f_r) \frac{(1-r)^{\alpha-1}}{r}dr\asymp\sum_{n\geq1} |\hat{f}(n)|^{2}(1+n)^{1-\alpha},
\end{equation}
which therefore yields
\begin{equation}\label{p3}
\cD_{\alpha}(f)\asymp\alpha\int_{0}^{1}\cD(f_r) \frac{(1-r)^{\alpha-1}}{r}dr,
 \end{equation}
independently of $\alpha$ and $f.$

This allows us to express $\cD_{\alpha}(f)$ in a way crucial for us. Indeed, the following lemma
reflects somehow the magic of the Cauchy-Riemann equations which allow to express the
weighted Dirichlet integral through an integration of a function which is not necessarily positive.
\begin{lem}\label{EqDirInt}
Let $f=e^{u+iv}\in \cD_{\alpha}$.
Set  
\[
dA_{\alpha}(z):=\alpha (1-r)^{\alpha-1}drd\theta,\qquad z:=re^{i\theta}\in\D.
\]
Then
\begin{equation}\label{motassa2}
\cD_{\alpha}(f)\asymp \frac{1}{2\pi}\int_{\D} |f(z)|^2\frac{\partial v}{\partial \theta}(z) dA_{\alpha}(z),
\end{equation}
independently of $\alpha$ and $f.$
\end{lem}

Observe the absence of the factor $r$ in the definition of $dA_{\alpha}$ which is thus
not the usual weighted area Lebesgue measure.

\begin{proof}
We begin reformulating $\cD(f_r)$.
Set $f=e^g$ with $g=u+iv$, then expressing first the derivative of $g$ in polar coordinates
and using then Cauchy Riemann equations, we get at $z=se^{i\theta}$,
\begin{eqnarray*}
 |f'|^2 
 &=& |f|^2\times |g'|^2 
 =|f|^2\times \Big|\frac{\partial u}{\partial s}+i\frac{\partial v}{\partial s}\Big|^2
 =|f|^2\times \left(\Big(\frac{\partial u}{\partial s}\Big)^2
 +\Big(\frac{\partial v}{\partial s}\Big)^2\right)\\
 &=&|f|^2\frac{1}{s}\Big(\frac{\partial u}{\partial s}\frac{\partial v}{\partial \theta}
 - \frac{\partial u}{\partial \theta}\frac{\partial v}{\partial s}\Big)
\end{eqnarray*}
On the other hand
\begin{eqnarray}\nonumber
\frac{\partial}{\partial s}\Big(|f|^2\frac{\partial v}{\partial \theta}\Big)
-\frac{\partial}{\partial \theta}\Big(|f|^2\frac{\partial v}{\partial s}\Big)
=2|f|^2\Big(\frac{\partial u}{\partial s}\frac{\partial v}{\partial \theta}
 - \frac{\partial u}{\partial \theta}\frac{\partial v}{\partial s}\Big),
\end{eqnarray}
so that replacing $f$ by $f_r$,
we get
\begin{eqnarray}\nonumber
2|f_r'(z)|^2s
= r\frac{\partial}{\partial s}\Big(|f|^2(rz)\frac{\partial v}{\partial \theta}(rz)\Big)
-r\frac{\partial}{\partial \theta}\Big(|f|^2(rz)\frac{\partial v}{\partial s}(rz)\Big).
\end{eqnarray}
Since $f$ is outer, the function $w\longmapsto |f|^2(rz)\frac{\partial v}{\partial s}(rz)$
is continuous on $\D$ so that
\[
\int_{0}^{2\pi}
 \frac{\partial}{\partial \theta}\Big(|f|^2(rz)\frac{\partial v}{\partial s}(rz)\Big) d\theta 
 =
 \Big[\Big(|f|^2(rse^{i\theta})\frac{\partial v}{\partial s}(rse^{i\theta})
 \Big)\Big]_{0^+}^{2\pi^-} 
=0,
\]
and hence
\begin{equation}\label{motassa1}
\cD(f_r)= \frac{1}{2\pi}\int_{-\pi}^{\pi} |f|^2(re^{i\theta})\frac{\partial v}{\partial \theta}(re^{i\theta}) rd\theta.
\end{equation}
Setting
\[
dA_{\alpha}(z):=\alpha (1-r)^{\alpha-1}drd\theta,\qquad z:=re^{i\theta}\in\D,
\]
we deduce \eqref{motassa2} from \eqref{p3} and \eqref{motassa1}
\end{proof}

As we have already mentioned in \eqref{estimNalpha} we have
$\cD_{\alpha}(O_{h})\gtrsim N_\alpha(h)$
independently of both $\alpha$ and $h,$ so in order to prove our main results we can suppose from now on that $N_\alpha(h)<+\infty$. 

Let $\T_{_h}$ be the set of points $e^{i\theta}\in\T$
where $O_h$ has radial boundary limit such that
$0<\displaystyle\lim_{r\rightarrow1^{-}}|O_h(re^{i\theta})|=h(e^{i\theta})<\infty.$
It is well known that $\T_{_h}$ coincides with $\T$ except for a set of zero Lebesgue measure.
We will also use the notations
$$\T_{_h}(\theta):=\{\varphi\in]-\pi,\pi]\ :\ h(\varphi)\stackrel{2}{\asymp} h(\theta)\},\qquad e^{i\theta}\in\T_{_h},$$
 $$\T_{_h}^{+}(\theta):=\{\varphi\in]-\pi,\pi]\ :\ h(\varphi)\geq 2h(\theta)\},\qquad e^{i\theta}\in\T_{_h},$$
 and
  $$\T_{_h}^{-}(\theta):=\{\varphi\in]-\pi,\pi]\ :\ h(\varphi)\leq\frac{1}{2}h(\theta)\},\qquad e^{i\theta}\in\T_{_h}.$$
We finally recall the following classical equality
\begin{equation}\label{circ555}
|e^{i\varphi}-z|^{2}=(1-r)^{2}+r|e^{i\varphi}-e^{i\theta}|^{2},\qquad z:=re^{i\theta}\in\D \text{ and }e^{i\varphi}\in\T,
\end{equation}
which yields the following estimate
\begin{equation}\label{circ5552}
|e^{i\varphi}-z|\geq\max\{1-r,\ \frac{1}{3}|e^{i\varphi}-e^{i\theta}|\},\qquad z=re^{i\theta}
 \in\D \text{ and }e^{i\varphi}\in\T.
\end{equation}

\section{\bf The function $\mu_{_h}.$ \label{muh}}


Recall that $\lambda a_{h,\lambda}$ and $\tilde{a}_{h,\lambda}/\lambda$ have interpretations
as Poisson integrals of $\log\frac{h(\theta)}{h(\varphi)}$ over $\T_h^-(\theta)$ and
$|e^{i\theta}-e^{i\varphi}|\ge\lambda(\theta)$ and 
$|e^{i\theta}-e^{i\varphi}|\le\lambda(\theta)$ respectively. The next lemma considers
the part of the Poisson integrals on $\T_h^+(\theta)$.

\begin{lem}\label{lem1}
Suppose $N_{\alpha}(h)<\infty$.
Then the Lebesgue measure of
$$\T_{h,\delta}:=\{ e^{i\theta}\in\T_{_h}\ :\ \limsup_{r\rightarrow 1^-}\frac{1}{2\pi}\int_{\T^{+}_{h}(\theta)}\frac{1-r^2}{|e^{i\varphi}-re^{i\theta}|^2}\log\frac{h(\varphi)}{h(\theta)}d\varphi
\geq\delta\}$$
is zero for every $\delta>0.$
\end{lem}

\begin{proof}
Let $0<\varepsilon\leq\frac{1}{2}$ be a real number. With each point $e^{i\theta}\in\T_{h,\delta}$ we associate $r_{\varepsilon}=
r_{\varepsilon,\theta}\in (0,1)$
such that
$r_\varepsilon\geq1-\varepsilon$ and
\begin{eqnarray}\label{chta}
\frac{1}{2\pi}\int_{\T^{+}_{h}(\theta)}\frac{1-r_\varepsilon^2}{|e^{i\varphi}-r_\varepsilon e^{i\theta}|^2}\log\frac{h(\varphi)}{h(\theta)}d\varphi
\geq\frac{\delta}{2}.
\end{eqnarray}
The dependence of $r_{\varepsilon}$ on $\theta$ is not relevant in the argument
below.
Using \eqref{circ5552}
\begin{eqnarray*}
 \frac{(1-r)^\alpha}{|e^{i\varphi}-re^{i\theta}|^2}
 \leq
 \left\{
 \begin{array}{ll}
 \frac{\dst 9|e^{i\varphi}-e^{i\theta}|^{\alpha}}{\dst |e^{i\varphi}-e^{i\theta}|^2}=\frac{\dst 9}{\dst |e^{i\varphi}-e^{i\theta}|^{2-\alpha}}
 & \text{if } |e^{i\varphi}-e^{i\theta}|\ge 1-r\\
  \frac{\dst 1}{\dst (1-r)^{2-\alpha}}\le\frac{\dst 1}{\dst |e^{i\varphi}-e^{i\theta}|^{2-\alpha}}
 & \text{if } |e^{i\varphi}-e^{i\theta}|<1-r
 \end{array}
\right..
\end{eqnarray*}
It follows
\begin{eqnarray}\label{ch}
\frac{(1-r)^\alpha}{|e^{i\varphi}-re^{i\theta}|^2}\leq \frac{\dst 9}{\dst |e^{i\varphi}-e^{i\theta}|^{2-\alpha}}.
\end{eqnarray}
Therefore, using \eqref{chta} and \eqref{ch},
\begin{eqnarray}\nonumber
 \int_{e^{i\theta}\in\T_{h,\delta}}\frac{h^{2}(\theta)}{\varepsilon^{1-\alpha}}d\theta
&\leq& \int_{ e^{i\theta}\in\T_{h,\delta}}\frac{h^{2}(\theta)}
  {(1-r_\varepsilon)^{1-\alpha}}d\theta
\\\nonumber &\leq& \frac{2}{\pi\delta}\int_{ e^{i\theta}\in\T_{h,\delta}}h^{2}(\theta) 
 \int_{\T^{+}_{h}(\theta)}
 \frac{(1-r_\varepsilon)^\alpha}{|e^{i\varphi}-r_\varepsilon e^{i\theta}|^2}\log\frac{h(\varphi)}{h(\theta)}d\varphi d\theta
\\\nonumber &\leq& \frac{18}{\pi\delta}\int_{e^{i\theta}\in\T_{h,\delta}}\int_{\T^{+}_{h}(\theta)}\frac{h^{2}(\theta)\log\frac{h(\varphi)}{h(\theta)}}{|e^{i\varphi}-e^{i\theta}|^{2-\alpha}}d\varphi d\theta
\\\label{bh}
&\leq& \frac{18}{\pi\delta}\int_{-\pi}^{\pi}\int_{-\pi}^{\pi}\frac{|h(\varphi)-h(\theta)|^2}{|e^{i\varphi}-e^{i\theta}|^{2-\alpha}}d\varphi d\theta  = \frac{18N_\alpha(h)}{\pi\delta}.
\end{eqnarray}
Since $h\neq0$ a.e.\ on $\T$, and letting $\varepsilon$ tend to 0, we deduce the desired result.
\end{proof}

We obtain the following lemma that provides some properties of $\mu_{_h}.$

\begin{lem}\label{lem2}
Suppose $N_{\alpha}(h)<\infty$. Then  $\mu_{_h}\in\Lambda$ and
\begin{eqnarray}\label{veb6}
|O_h(re^{i\theta})|\geq e^{-41}h(\theta), \qquad r\geq1-\mu_{_h}(\theta),
\end{eqnarray}
for every point $e^{i\theta}\in\T_{_{h}}$ such that $\mu_{_h}(\theta)>0.$
If  $e^{i\theta}\in\T_{_{h}}$ is a point such that $0<\mu_{_h}(\theta)<1,$ then
\begin{eqnarray}\label{veb4}
\frac{1}{2\pi}\int_{\T^{-}_{_h}(\theta)}\frac{1-|z_h(\theta)|^2}{|e^{i\varphi}-z_h(\theta)|^2}\log\frac{h(\theta)}{h(\varphi)} d\varphi \geq 1,
\end{eqnarray}
where $z_h(\theta):=(1-\mu_{_h}(\theta))e^{i\theta}\in\D\setminus\{0\}.$
\end{lem}

\begin{proof}
In order to check that $\mu_{_h}\in\Lambda$,
we need to show that $\mu_{_h}$ is strictly positive almost everywhere.
Suppose $\mu_{_h}(\theta)=0$ for a fixed point
$e^{i\theta}\in\T_{_h}$, i.e.\
there exists a sequence of positive numbers $\{\delta_n:n\in\N\}\subset]0,1]$  converging to $0$ and satisfying, for each $n\in\N,$ at least one of the following inequalities
\begin{equation}\label{lisa2}
\frac{\widetilde{a}_{_{h,\delta_n}}(\theta)}{\delta_n}
=
\frac{1}{2\pi\delta_n}\int_{\T^{-}_{_h}(\theta)\atop |e^{i\varphi}-e^{i\theta}|\leq \delta_n}{\log \frac{h(\theta)}{h(\varphi)}} d\varphi >2,
\end{equation}
or
\begin{equation}\label{lisa1}
\delta_n a_{_{h,\delta_n}}(\theta)=
\frac{\delta_n}{2\pi}\int_{\T^{-}_{_h}(\theta)\atop |e^{i\varphi}-e^{i\theta}|\geq \delta_n}\frac{\log \frac{h(\theta)}{h(\varphi)}}{|e^{i\varphi}-e^{i\theta}|^2} d\varphi >2.
\end{equation}
Associated to $e^{i\theta}$ and the numbers $\delta_n,$ we define in $\D$ the following points   $$z_n:=(1-\delta_n)e^{i\theta},\qquad n\in\N.$$
Since $e^{i\theta}\in\T_{_h}$ and $\lim\limits_{n\rightarrow \infty}\delta_n=0,$ there exists a number $N_0\in\N$ such that
$$|\log\frac{|O_h(z_n)|}{h(\theta)}|\leq \frac{1}{4},\qquad \text{ for all }n\geq N_0.$$
Since $\delta_n=1-|z_n|$ and using \eqref{lisa2} and \eqref{lisa1}
\begin{equation}\label{lis}
\frac{1}{2\pi}\int_{\T^{-}_{_h}(\theta)}\frac{1-|z_n|^2}{|e^{i\varphi}-z_n|^2}\log\frac{h(\theta)}{h(\varphi)} d\varphi >1.
\end{equation}
By decomposition
\begin{eqnarray} \nonumber
&&\frac{1}{2\pi}\int_{\T^{+}_{h}(\theta)}\frac{1-|z_n|^2}{|e^{i\varphi}-z_n|^2}\log\frac{h(\varphi)}{h(\theta)} d\varphi
\\\nonumber&=&\log\frac{|O_h(z_n)|}{h(\theta)}
-\frac{1}{2\pi}\int_{\T_{h}(\theta)}\frac{1-|z_n|^2}{|e^{i\varphi}-z_n|^2}\log\frac{h(\varphi)}{h(\theta)} d\varphi
\\\label{veb1}&&+ \frac{1}{2\pi} \int_{\T^{-}_{_h}(\theta)}\frac{1-|z_n|^2}{|e^{i\varphi}-z_n|^2}\log\frac{h(\theta)}{h(\varphi)} d\varphi
\end{eqnarray}
(observe the inversion of the $\log$-fraction in the last integral explaining
the plus-sign before this integral)
and
\begin{eqnarray}\label{veb2}
\frac{1}{2\pi}\int_{\T_{h}(\theta)}\frac{1-|z_n|^2}{|e^{i\varphi}-z_n|^2}|\log\frac{h(\varphi)}{h(\theta)}| d\varphi \leq \log 2,
\end{eqnarray}
so that
\begin{eqnarray*}
\lefteqn{\frac{1}{2\pi}\int_{\T^{+}_{h}(\theta)}\frac{1-|z_n|^2}{|e^{i\varphi}-z_n|^2}\log\frac{h(\varphi)}{h(\theta)} d\varphi}\\
&&\geq \frac{1}{2\pi}\int_{\T^{-}_{_h}(\theta)}\frac{1-|z_n|^2}{|e^{i\varphi}-z_n|^2}\log\frac{h(\theta)}{h(\varphi)} d\varphi -|\log\frac{|O_h(z_n)|}{h(\theta)}|\\
&&\quad -\frac{1}{2\pi}\int_{\T_{h}(\theta)}\frac{1-|z_n|^2}{|e^{i\varphi}-z_n|^2}|\log\frac{h(\varphi)}{h(\theta)}| d\varphi\\
&&\geq\frac{3}{4}-\log 2, \qquad n\geq N_0.
\end{eqnarray*}
So, $e^{i\theta}\in \T_{h,3/4-\log 2}$, and, by Lemma \ref{lem1}, $\mu_{_h}>0$ a.e.\ on $\T$,
and thus $\mu_{_h}\in\Lambda.$

Now, we let $z\in\D$ be a point such that $r\geq 1-\mu_{_h}(\theta),$ then
\begin{eqnarray*}
&&\frac{1}{2\pi}\int_{ h(\varphi)\leq h(\theta)}\frac{1-r^2}{|e^{i\varphi}-z|^2}\log\frac{h(\theta)}{h(\varphi)} d\varphi\\
&=& \frac{1}{2\pi}\int_{ \frac{1}{2}h(\theta)\leq h(\varphi)\leq h(\theta)} + \frac{1}{2\pi}\int_{\T^{-}_{_h}(\theta)\atop|e^{i\varphi}-e^{i\theta}|\leq 1-r}
+\frac{1}{2\pi}\int_{\T^{-}_{_h}(\theta)\atop|e^{i\varphi}-e^{i\theta}|\geq 1-r}\\
&\leq& \log 2+ \frac{1}{\pi}\int_{\T^{-}_{_h}(\theta)\atop|e^{i\varphi}-e^{i\theta}|\leq 1-r}\frac{\log\frac{h(\theta)}{h(\varphi)}}{1-r} d\varphi +
\frac{9}{\pi}(1-r)\int_{\T^{-}_{_h}(\theta)\atop|e^{i\varphi}-e^{i\theta}|\geq 1-r} \frac{\log\frac{h(\theta)}{h(\varphi)}}{|e^{i\varphi}-e^{i\theta}|^2} d\varphi,
\end{eqnarray*}
where we have used \eqref{ch}. By the very definition of $\mu_{_h}(\theta)$, this yields
\begin{eqnarray}\label{veb5}
\frac{1}{2\pi}\int_{ h(\varphi)\leq h(\theta)}\frac{1-r^2}{|e^{i\varphi}-z|^2}\log\frac{h(\theta)}{h(\varphi)} d\varphi
\le \log 2+ 4+36.
\end{eqnarray}
Since obviously
\[
\frac{1}{2\pi}\int_{ h(\theta)\leq h(\varphi) }\frac{1-r^2}{|e^{i\varphi}-z|^2}\log\frac{h(\theta)}{h(\varphi)} d\varphi \leq 0,
\]
we obtain \eqref{veb6}. We argue similarly as in the proof of \eqref{lis} to show that if $0<\mu_{_h}(\theta)<1$ then there exists a sequence of positive numbers
$\{\varepsilon_n:n\in\N\}\subset]0,1]$  converging to $0$ such that
\begin{equation}\label{lis2}
\frac{1}{2\pi}\int_{\T^{-}_{_h}(\theta)}\frac{1-|w_n|^2}{|e^{i\varphi}-w_n|^2}\log\frac{h(\theta)}{h(\varphi)} d\varphi \geq1,
\end{equation}
where $w_n:= (1-(\mu_{_h}(\theta)+\varepsilon_n))e^{i\theta}\in\D\setminus\{0\}.$ We apply Lebesgue's dominated convergence theorem in \eqref{lis2} to deduce \eqref{veb4}.
\end{proof}

The following Lemma gives a lower estimate of $\cD_{\alpha}(O_h)$ involving $\mu_{_h}$, and
will be used in Section \ref{sect4} to get some necessary conditions for
$O_h\in\cD_\alpha.$

\begin{lem} \label{dauglas}
We have
$$\int_{\mu_{_h}(\theta)<1}h^{2}(\theta)\mu^{\alpha-1}_{h}(\theta)d\theta\leq c \cD_{\alpha}(O_h),$$
where $c>0$ is a constant independent of both $\alpha$ and $h.$
\end{lem}

Adding the points where $\mu_{_h}(\theta)=1$, we get
\begin{eqnarray}\label{Lem3.3bis}
\int_{\T}h^{2}(\theta)\mu^{\alpha-1}_{h}(\theta)d\theta\leq c \big(\cD_{\alpha}(O_h)
 +\|h\|_2^2\big).
\end{eqnarray}

\begin{proof}  According to Lemma \ref{lem2}, we have $\mu_{_h}\in\Lambda$.
Let $e^{i\theta_{_0}}\in\T_{_h}$ be a point such that $0<\mu_{_h}(\theta_{_0})<1.$ For the point $z_h(\theta_{_0})=(1-\mu_{_h}(\theta_0))e^{i\theta_0}$, we claim that two cases may occur:
\begin{equation}\label{case1}
\frac{1}{2\pi}\int_{\T^{+}_{h}(\theta_{_0})}\frac{1-|z_h(\theta_{_0})|^2}{|e^{i\varphi}-z_h(\theta_{_0})|^2}\log\frac{h(\varphi)}{h(\theta_{_0})} d\varphi
\geq\frac{1}{4}
\end{equation}
or
\begin{equation}\label{case2}
\frac{|O_h(z_h(\theta_{_0}))|}{h(\theta_{_0})}\leq e^{\log 2-\frac{3}{4}}.
\end{equation}
Indeed, if we suppose that \eqref{case1} is false, then
with \eqref{veb4}
\begin{eqnarray}\nonumber
\log \frac{|O_h(z_h(\theta_{_0}))|}{h(\theta_{_0})}
&=&\frac{1}{2\pi}\int_{\T^{-}_{h}(\theta_{_0})}\frac{1-|z_h(\theta_{_0})|^2}{|e^{i\varphi}-z_h(\theta_{_0})|^2}\log\frac{h(\varphi)}{h(\theta_{_0})} d\varphi
\\\nonumber&&+\frac{1}{2\pi}\int_{\T^{+}_{h}(\theta_{_0})}\frac{1-|z_h(\theta_{_0})|^2}{|e^{i\varphi}-z_h(\theta_{_0})|^2}\log\frac{h(\varphi)}{h(\theta_{_0})} d\varphi
\\\nonumber&&+\frac{1}{2\pi}\int_{\T_{h}(\theta_{_0})}\frac{1-|z_h(\theta_{_0})|^2}{|e^{i\varphi}-z_h(\theta_{_0})|^2}\log\frac{h(\varphi)}{h(\theta_{_0})} d\varphi
\\\nonumber&\leq&-1+\frac{1}{4}+\log2,
\end{eqnarray}
which shows \eqref{case2}.

Now, on the one hand, if $\theta_{_0}$ satisfies \eqref{case1}, then (with \eqref{ch} in mind),
\begin{eqnarray}\nonumber
\mu^{\alpha-1}_{h}(\theta_{_0})&=& (1-|z_h(\theta_0)|^2)^{\alpha-1}
\leq \frac{4}{\pi}\int_{\T^{+}_{h}(\theta_{_0})}\frac{(1-|z_h(\theta_{_0})|)^{\alpha}}{|e^{i\varphi}-z_h(\theta_{_0})|^2}\log\frac{h(\varphi)}{h(\theta_{_0})} d\varphi
\\\nonumber&\lesssim&\int_{\T^{+}_{h}(\theta_{_0})}\frac{\log\frac{h(\varphi)}{h(\theta_{_0})}}{|e^{i\varphi}-e^{i\theta_{_0}}|^{2-\alpha}} d\varphi ,
\end{eqnarray}
which gives, using \eqref{logequiv} and the triangular inequality,
\begin{eqnarray}\label{case11}
h^{2}(\theta_{_0})\mu^{\alpha-1}_{h}(\theta_{_0})\lesssim \int_{-\pi}^{\pi}\frac{|h(\varphi)-h(\theta_{_0})|^2}{|e^{i\varphi}-e^{i\theta_{_0}}|^{2-\alpha}} d\varphi 
 \le \int_{-\pi}^{\pi}\frac{|O_h(\varphi)-O_h(\theta_{_0})|^2}
 {|e^{i\varphi}-e^{i\theta_{_0}}|^{2-\alpha}} d\varphi,
\end{eqnarray}
independently of $\theta_{_0},$ $\alpha$ and $h.$

On the other hand, for almost all points $e^{i\theta_{_0}}\in\T_{_h}$ that satisfy \eqref{case2}, we have
$h^2(\theta_{_0})\lesssim|O_h(z_h(\theta_{_0}))-O_h(\theta_{_0})|^2$(observe that
$e^{\log 2-3/4}<1$), and thus, by Jensen's inequality,
\begin{eqnarray*}
h^2(\theta_{_0})\lesssim
\frac{1}{2\pi}\int_{-\pi}^{\pi}\frac{1-|z_h(\theta_{_0})|^2}{|e^{i\varphi}- z_h(\theta_{_0})|^{2}}|O_h(\varphi)-O_h(\theta_{_0})|^2 d\varphi.
\end{eqnarray*}
As a consequence
\begin{eqnarray}\nonumber
h^{2}(\theta_{_0})\mu^{\alpha-1}_{h}(\theta_{_0})&\lesssim &\int_{-\pi}^{\pi}\frac{(1-|z_h(\theta_{_0})|)^\alpha}{|e^{i\varphi}- z_h(\theta_{_0})|^{2}}|O_h(\varphi)-O_h(\theta_{_0})|^2 d\varphi
\\\label{case12}&\lesssim&\int_{-\pi}^{\pi}\frac{|O_h(\varphi)-O_h(\theta_{_0})|^2}{|e^{i\varphi}-e^{i\theta_{_0}}|^{2-\alpha}} d\varphi ,
\end{eqnarray}
independently of $\theta_{_0},$ $\alpha$ and $h.$
The desired result follows from Douglas' formula \eqref{DouglasFormulaalpha} and the inequalities \eqref{case11} and \eqref{case12}.
\end{proof}

\section{\bf The sufficiency. \label{sect3}}

In this section we prove the sufficient condition of Theorem \ref{ra},
more precisely
\begin{equation}\label{sufficient}
\cD_{\alpha}(O_h)\lesssim N_{\alpha}(h)+ \frac{1}{1-\alpha} \inf_{\lambda\in\Lambda}\Big( n_{\alpha}(h,\lambda)+\widetilde{n}_{\alpha}(h,\lambda)\Big).
\end{equation}
Observe that for this upper estimate we do not need the term $\|O_h\|_2^2$.

Recall from Lemma \ref{EqDirInt} that in order to prove that $O_h=e^{u+iv}\in\mathcal{D}_{\alpha}$
it is sufficient to estimate the integral
\begin{eqnarray}\label{IntEstim}
\frac{1}{2\pi}\int_{\D} |O_h(z)|^2\frac{\partial v_{_h}}{\partial \theta}(z) dA_{\alpha}(z),
\end{eqnarray}
where $dA_{\alpha}(z)=\alpha(1-r)^{\alpha-1}drd\theta$.

Depending on $h,$ we define the following set of rays
 $$\D_{_h}:=\{z\in\D :\ e^{i\theta}\in\T_{_h}\},$$
which we divide into the following two parts
$$\K_{_h}:= \big\{z=re^{i\theta}\in\D_{_h}\ :\  \sup\limits_{w\in\D(z)}|O_{h}(w)|\geq 2h(\theta)\big\},$$
where $\D(z):=\{w\in\D\ :\ |w-z|\leq \frac{1}{2}(1-r)\}$ is a
pseudohyperbolic disk with fixed radius,
and
$$\LL_{_h}:=\D_{_h}\setminus\K_{_h}.$$
Observe that we do not need to consider integration on the remainder
set $\D\setminus (\K_h\cup\LL_h)$ which is a union --- over a set of Lebesgue measure
zero on $\T$ --- of rays and hence
of Lebesgue area measure zero.

\subsection{The integration on the region $\K_h$} \label{Kh}
In the following Lemma we show that the integral on $\K_h$ is controlled by $N_{\alpha}(h)$ only.

\begin{lem}\label{harmonic}
We have
\begin{eqnarray}\nonumber
\int_{\K_{h}}\big|O_{h}^{2}(z)\frac{\partial v}{\partial\theta}(z)\big|dA_{\alpha}(z)
\leq c N_{\alpha}(h),
\end{eqnarray}
where $c>0$ is a constant independent of both $\alpha$ and $h.$
\end{lem}

\begin{proof}
We suppose that the area Lebesgue measure of  $\K_h$ is different from zero
(otherwise there is nothing to prove).
Clearly
\begin{equation}\label{asia1}
\sup\limits_{w\in\D(z)}|O_{h}(w)|\leq 2 \sup\limits_{w\in\D(z)}\big||O_{h}(w)|-h(\theta)\big|,  \qquad z\in\K_{_h}.
\end{equation}
We set
$$\cH (z):= \frac{1}{2\pi}\int_{-\pi}^{\pi}\frac{1-r^2}{|e^{i\varphi}-z|^{2}}\big|h(\varphi)-h(\theta)\big| d\varphi ,\qquad z=re^{i\theta}\in\D_{_h}.$$
For a point $z=re^{i\theta}\in\K_{_h}$
\begin{eqnarray}\nonumber
\sup\limits_{w\in\D(z)}\big||O_{h}(w)|-h(\theta)\big|&=&\sup\limits_{w\in\D(z)}\{|O_{h}(w)|\}-h(\theta)
\\\nonumber&\leq&\sup\limits_{w\in\D(z)}\Big\{\frac{1}{2\pi}\int_{-\pi}^{\pi}\frac{1-|w|^2}{|e^{i\varphi}-w|^{2}}h(\varphi) d\varphi \Big\} -h(\theta)
\\\nonumber&\leq&\sup\limits_{w\in\D(z)}\Big\{\frac{1}{2\pi}\int_{-\pi}^{\pi}\frac{1-|w|^2}{|e^{i\varphi}-w|^{2}}\big|h(\varphi)-h(\theta)\big| d\varphi  \Big\}
\\\nonumber&\leq&\frac{6}{\pi}\int_{-\pi}^{\pi}\frac{1-r^2}{|e^{i\varphi}-z|^{2}}\big|h(\varphi)-h(\theta)\big| d\varphi
\\\label{asia2}
&=& 12 \cH (z).
\end{eqnarray}
Hence, with \eqref{asia1} 
we get
\begin{eqnarray}\label{circ}
\sup\limits_{w\in\D(z)}|O_{h}(w)|\leq 24\cH (z),\qquad z\in\K_{_h}.
\end{eqnarray}
The classical Cauchy formula for holomorphic functions
applied to the complex derivative of $O_{h}$ on $\partial \D(z)$ implies
\begin{eqnarray}\label{cauchy}
|\frac{\partial O_{h}}{\partial z}(z)|
\leq 2\frac{\sup\limits_{w\in\D(z)}|O_{h}(w)|}{1-r},\qquad z\in\D,
\end{eqnarray}
so that
\begin{eqnarray*}
\big|\frac{\partial O_{h}^{2}}{\partial z}(z)\big|
=\big|2O_h(z)\frac{\partial O_{h}}{\partial z}(z)\big|
\leq\label{circ1} 48^2\frac{\cH^{2} (z)}{1-r},\qquad z\in\K_{_h}.
\end{eqnarray*}
Jensen's inequality
implies
\begin{eqnarray*}
\cH^{2} (z)\leq \frac{1}{2\pi}\int_{-\pi}^{\pi}\frac{1-r^2}{|e^{i\varphi}-z|^{2}}\big|h(\varphi)-h(\theta)\big|^{2} d\varphi ,\qquad z\in\D_{_h},
\end{eqnarray*}
which gives
\begin{eqnarray*}
\big|\frac{\partial O_{h}^{2}}{\partial z}(z)\big|\leq
\frac{48^2}{\pi}\int_{-\pi}^{\pi}\frac{\big|h(\varphi)-h(\theta)\big|^{2}}{|e^{i\varphi}-z|^{2}} d\varphi ,\qquad z\in\K_{_h}.
\end{eqnarray*}
Since $O_h=e^g=e^{u+iv}$, and $g'(re^{i\theta})=\frac{e^{-i\theta}}{r}\Big(
\frac{\partial v}{\partial \theta}(re^{i\theta})-i\frac{\partial u}{\partial \theta}(re^{i\theta})\Big)$,
a computation yields
\begin{eqnarray}\label{circ5}
 |\frac{\partial O_h^2}{\partial z}(z)|&=&|2O_h(z)\frac{\partial O_h}{\partial z}(z)|
 =\frac{2|O_h(z)|^2}{r}\sqrt{\Big(\frac{\partial v}{\partial \theta}(z)\Big)^2+
 \Big(\frac{\partial u}{\partial \theta}(z)\Big)^2}\nonumber\\
&\ge& 
\frac{2}{r}|O_h(z)|^2|\frac{\partial v_h}{\partial \theta}(z)|,\qquad z\in\D,
\end{eqnarray}
and hence
\begin{equation}\label{circ55}
\big|O_{h}^{2}(z)\frac{\partial v_{_h}}{\partial\theta}(z)\big|\leq
\frac{48^2}{2\pi}\int_{-\pi}^{\pi}\frac{\big|h(\varphi)-h(\theta)\big|^{2}}{|e^{i\varphi}-z|^{2}}r d\varphi ,\qquad z\in\K_{_h}.
\end{equation}
With $z=re^{i\theta}$ and using \eqref{circ5552} we get
\begin{eqnarray}\label{circ6}
\lefteqn{\alpha\int_{0}^{1}\frac{(1-r)^{\alpha-1}}{|e^{i\varphi}-z|^{2}}dr}
  \nonumber \\
&&=\alpha\int_{0}^{1-\frac{1}{2}|e^{i\varphi}-e^{i\theta}|}\frac{(1-r)^{\alpha-1}}
 {|e^{i\varphi}-re^{i\theta}|^{2}}dr
 +\alpha\int_{1-\frac{1}{2}|e^{i\varphi}-e^{i\theta}|}^{1}\frac{(1-r)^{\alpha-1}}
 {|e^{i\varphi}-re^{i\theta}|^{2}}dr \nonumber  \\
&&\leq \alpha\int_{0}^{1-\frac{1}{2}|e^{i\varphi}-e^{i\theta}|}\frac{(1-r)^{\alpha-1}}{(1-r)^{2}}dr
 +\frac{9\alpha}{|e^{i\varphi}-e^{i\theta}|^{2}}\int_{1-\frac{1}{2}|e^{i\varphi}-e^{i\theta}|}^{1}{(1-r)^{\alpha-1}}dr\nonumber \\ 
&&\lesssim \frac{1}{|e^{i\varphi}-e^{i\theta}|^{2-\alpha}}, \qquad e^{i\varphi}\in\T\setminus\{e^{i\theta}\}.
\end{eqnarray}
Hence,  inequalities \eqref{circ55} and \eqref{circ6} yield
\begin{eqnarray}\nonumber
\int_{\K_{h}}\big|O_{h}^{2}(z)\frac{\partial v_{_h}}{\partial\theta}(z)\big|dA_{\alpha}(z)
&\leq&\frac{48^2}{2\pi}\int_{-\pi}^{\pi}\Big(\int_{\D}\big|\frac{h(\varphi)-h(\theta)}{e^{i\varphi}-z}\big|^{2}  dA_{\alpha}(z)\Big)d\varphi
\\\label{intkh}&\leq& c N_{\alpha}(h),
\end{eqnarray}
where $c>0$ is a constant independent of both $\alpha$ and $h.$
\end{proof}

\subsection{The integration on the region $\LL_{h}$}\label{Kh1}

Recall that by definition
\[
 \LL_h=\{z\in\D_h:\sup_{w\in \D(z)} |O_h(w)|<2h(\theta)\}.
\]
Fix $\lambda\in\Lambda.$ Associated with $h$ and $\lambda$ we define the following function
$$\rho_{h,\lambda}(\theta):=\min\{\mu_{_h}(\theta),2a^{-1}_{_{h,\lambda}}(\theta)\},\qquad e^{i\theta}\in\T_{_h}.$$
Since $\lambda\in\Lambda$,
a simple estimate of the integral in \eqref{ahlambda} shows
that $a_{_{h,\lambda}}<+\infty$ a.e.\ with respect to Lebesgue measure, and hence  $\rho_{h,\lambda}\in\Lambda.$

In order to estimate our integral over the region $\LL_{h}$ we need to
divide it into the following two parts
$$\LL^{1}_{h,\lambda}:=\{z\in\LL_{h}\ :\ r\leq1-\rho_{h,\lambda}(\theta)\}$$
and
$$\LL^{2}_{h,\lambda}:=\{z\in\LL_{h}\ :\ r\geq 1-\rho_{h,\lambda}(\theta)\}.$$
We observe here that since $\rho_{h,\lambda}\in\Lambda$
the boundary
of $\LL^{1}_{h,\lambda}$ meets $\T$ on a set of zero Lebesgue measure
while for $\LL^{2}_{h,\lambda}$ this happens on a set of full measure.

\subsubsection{\bf The integration on the region $\LL^{1}_{h,\lambda}$.}
In this section we
discuss the control of the integral in \eqref{IntEstim} on $\LL_{h,\lambda}^1$.

\begin{lem}\label{l1h}
\begin{eqnarray}\nonumber
\int_{\LL^{1}_{h,\lambda}}\big|O_{h}^{2}(z) \frac{\partial v_{_h}}{\partial\theta}(z)\big| dA_\alpha(z)
\leq\frac{c\alpha}{1-\alpha}(\widetilde{n}_\alpha(h,\lambda)+n_\alpha(h,\lambda)),
\end{eqnarray}
where $c>0$ is a constant independent of $\alpha,$ $h$ and $\lambda.$
\end{lem}

\begin{proof}
From \eqref{cauchy}, \eqref{circ5}  and the very definition of
$\LL_h$,
\begin{eqnarray*}
\big|O^{2}_{h}(z)\frac{\partial v_{_h}}{\partial\theta}(z)\big|
\leq \frac{r}{2}\left|2O_h(z)\frac{\partial O_h}{\partial z}(z)\right|
\le 8r \frac{h^{2}(\theta)}{1-r}, \qquad z\in\LL_{h}.
\end{eqnarray*}
It follows that
\begin{eqnarray}\nonumber
\int_{\LL^{1}_{h,\lambda}}\big|O_{h}^{2}(z) \frac{\partial v_{_h}}{\partial\theta}(z)\big| dA_\alpha(z)
&\leq& 8\alpha \int_{-\pi}^{\pi}h^{2}(\theta)\big(\int_{0}^{1-\rho_{h,\lambda}(\theta)}(1-r)^{\alpha-2}dr\big)d\theta
\\\nonumber
 &=& \frac{8\alpha}{1-\alpha}\int_{-\pi}^{\pi}h^{2}(\theta)\rho_{h,\lambda}^{\alpha-1}(\theta) d\theta
-\frac{8\alpha}{1-\alpha}\int_{-\pi}^{\pi}h^{2}(\theta)d\theta
\\\label{3anba3}&\leq& \frac{8\alpha}{1-\alpha}\int_{\rho_{h,\lambda}(\theta)<1}h^{2}(\theta)\rho_{h,\lambda}^{\alpha-1}(\theta) d\theta.
\end{eqnarray}
We let $e^{i\theta_{_0}}\in\T_{_h}$ be a point such that $0<\rho_{h,\lambda}(\theta_{_0})<1.$
We first suppose that  $\rho_{h,\lambda}(\theta_{_0})=2a_{h,\lambda}^{-1}(\theta_{_0}).$ Then 
\begin{eqnarray}\label{3anba31}
\rho_{h,\lambda}^{\alpha-1}(\theta_{_0})=( a_{h,\lambda}(\theta_{_0})/2)^{1-\alpha}
\leq a_{h,\lambda}^{1-\alpha}(\theta_{_0}).
\end{eqnarray}

Now we assume that $\rho_{h,\lambda}(\theta_{_0})= \mu_{_h}(\theta_{_0}).$   Then by \eqref{veb4}
\begin{eqnarray*}
\rho_{h,\lambda}^{\alpha-1}(\theta_{_0})
 =(1-|z_h(\theta_{_0})|)^{\alpha-1}
\leq  \frac{1}{\pi}\int_{\T^{-}_{h}(\theta_{_0})}\frac{(1-|z_h(\theta_{_0})|)^\alpha}{|e^{i\varphi}-z_h(\theta_{_0})|^2}\log\frac{h(\theta_{_0})}{h(\varphi)} d\varphi .
\end{eqnarray*}
Therefore (with \eqref{ch} and \eqref{circ5552} in mind)
\begin{eqnarray*}\nonumber
\rho_{h,\lambda}^{\alpha-1}(\theta_{_0})
&\leq&  \frac{1}{\pi}\int_{\T^{-}_{h}(\theta_{_0})\atop|e^{i\varphi}-e^{i\theta_{_0}}|\leq\lambda(\theta_{_0})}\frac{(1-|z_h(\theta_{_0})|)^\alpha}{|e^{i\varphi}-z_h(\theta_{_0})|^2}\log\frac{h(\theta_{_0})}{h(\varphi)} d\varphi
\\\nonumber&&+  \frac{1}{\pi}\int_{\T^{-}_{h}(\theta_{_0})\atop|e^{i\varphi}-e^{i\theta_{_0}}|\geq\lambda(\theta_{_0})}\frac{(1-|z_h(\theta_{_0})|)^\alpha}{|e^{i\varphi}-z_h(\theta_{_0})|^2}\log\frac{h(\theta_{_0})}{h(\varphi)} d\varphi
\\\nonumber&\leq&  \frac{9}{\pi}\int_{\T^{-}_{h}(\theta_{_0})\atop|e^{i\varphi}-e^{i\theta_{_0}}|\leq\lambda(\theta_{_0})}\frac{\log\frac{h(\theta_{_0})}{h(\varphi)}}{|e^{i\varphi}-e^{i\theta_{_0}}|^{2-\alpha}} d\varphi
\\&&+
\frac{9}{\pi}(1-|z_h(\theta_{_0})|)^\alpha\int_{\T^{-}_{h}(\theta_{_0})\atop|e^{i\varphi}-e^{i\theta_{_0}}|\geq\lambda(\theta_{_0})}
\frac{\log\frac{h(\theta_{_0})}{h(\varphi)}}{|e^{i\varphi}-e^{i\theta_{_0}}|^2} d\varphi .
\end{eqnarray*}
By our assumption $\rho_{h,\lambda}(\theta_{_0})=\mu_{_h}(\theta_{_0})\le 2a^{-1}_{h,\lambda}(\theta_{_0})$, so
that $(1-|z_h(\theta_{_0})|)^{\alpha}=\mu_{_h}^{\alpha}(\theta_{_0})\le
2 a_{h,\lambda}^{-\alpha}(\theta_{_0})$.
Therefore, by the very definition of $a_{h,\lambda}(\theta)$,
\[
\frac{9}{\pi}(1-|z_h(\theta_{_0})|)^\alpha\int_{\T^{-}_{h}(\theta_{_0})\atop|e^{i\varphi}-e^{i\theta_{_0}}|\geq\lambda(\theta_{_0})}\frac{\log\frac{h(\theta_{_0})}{h(\varphi)}}{|e^{i\varphi}-e^{i\theta_{_0}}|^2} d\varphi
\le 36 a_{h,\lambda}^{1-\alpha}(\theta_{_0}).
\]
Hence
\begin{eqnarray}\label{3anba4}
\rho_{h,\lambda}^{\alpha-1}(\theta_{_0})
\leq  \frac{9}{\pi}\int_{\T^{-}_{h}(\theta_{_0})\atop|e^{i\varphi}-e^{i\theta_{_0}}|\leq\lambda(\theta_{_0})}\frac{\log\frac{h(\theta_{_0})}{h(\varphi)}}{|e^{i\varphi}-e^{i\theta_{_0}}|^{2-\alpha}} d\varphi
+
36 a_{h,\lambda}^{1-\alpha}(\theta_{_0}).
\end{eqnarray}
By combining \eqref{3anba3}, \eqref{3anba31} and \eqref{3anba4} we deduce
\begin{eqnarray}\nonumber
&&\int_{\LL^{1}_{h,\lambda}}\big|O_{h}^{2}(z) \frac{\partial v_{_h}}{\partial\theta}(z)\big| dA_\alpha(z)
\\\nonumber&\lesssim&
\frac{\alpha}{1-\alpha}\Big(\int_{-\pi}^{\pi}h^{2}(\theta)\big(\int_{\T^{-}_{_h}(\theta)\atop|e^{i\varphi}-e^{i\theta}|\leq\lambda(\theta)}\frac{\log\frac{h(\theta)}{h(\varphi)}}{|e^{i\varphi}-e^{i\theta}|^{2-\alpha}} d\varphi\big)d\theta
+\int_{-\pi}^{\pi}h^{2}(\theta)a_{_{h,\lambda}}^{1-\alpha}(\theta)d\theta\Big)
\\&\asymp&\label{intl1h}\frac{\alpha}{1-\alpha}(\widetilde{n}_\alpha(h,\lambda)+n_\alpha(h,\lambda)),
\end{eqnarray}
where $c>0$ is a constant independent of $\alpha,$ $h$ and $\lambda.$
\end{proof}

\subsubsection{\bf The integration on the region $\LL^{2}_{h,\lambda}$.}
The estimates on this domain are more complicated. It is actually not possible to use the
triangular inequality directly, and some symmetry properties of the derivative of
the conjugate Poisson kernel need to be exploited in the estimates of
$\frac{\dst \partial v_h}{\dst \partial \theta}$.
To be more precise, we need to recall that $v_{_h}$ is the conjugate function of $u_{_h}$:
\begin{eqnarray*}
 v_{_h}(z)&:=&\frac{1}{2\pi}\int_{-\pi}^{\pi}\operatorname{Im}
 \left(\frac{e^{i\varphi}+z}{e^{i\varphi}-z}
 \right)\log h(\varphi) d\varphi , \qquad z\in\D.
\end{eqnarray*}
Observe that the function
\[
 Q(e^{i\varphi},z):=\frac{\partial}{\partial\theta}\operatorname{Im}\left(\frac{e^{i\varphi}+z}{e^{i\varphi}-z}
 \right),\qquad z=re^{i\theta}\in\D,
\]
 depends only on
$|e^{i\varphi}-e^{i\theta}|$ and $r.$ More precisely, we have
\begin{eqnarray}\label{Q}
Q(e^{i\varphi},z)&=&r\frac{2(1-r)^2-|e^{i\theta}-e^{i\varphi}|^{2}(1+r^2)}{|e^{i\varphi}-z|^4},
\end{eqnarray}
which yields the following estimate
\begin{equation}\label{qu2}
|Q(e^{i\varphi},z)|\leq \frac{2}{|e^{i\varphi}-z|^2}.
\end{equation}
Note also that $\int_{-\pi}^{\pi}Q(e^{i\varphi},z)d\varphi=0$, and hence
\begin{eqnarray}\label{log}
\frac{\partial v_{_h}}{\partial\theta}(z)
=\frac{1}{2\pi}\int_{-\pi}^{\pi}Q(e^{i\varphi},z)\log \frac{h(\varphi)}{h(\theta)} d\varphi ,\qquad z\in\D_{_h}.
\end{eqnarray}
In particular
\begin{eqnarray}\label{estimdv}
\Big|\frac{\partial v_{_h}}{\partial\theta}(z)\Big|
\le\frac{1}{\pi}\int_{-\pi}^{\pi}\frac{|\log \frac{h(\varphi)}{h(\theta)}|}{|e^{i\varphi}-z|^2}dt.
\end{eqnarray}

\begin{lem}\label{l2h}We have
\begin{eqnarray}\nonumber
\Big|\int_{\LL^{2}_{h,\lambda}}|O_{h}^{2}(z)| \frac{\partial v_{_h}}{\partial\theta}(z)dA_{\alpha}(z)\Big|
\leq c N_{\alpha}(h)+\frac{c}{1-\alpha}(\widetilde{n}_\alpha(h,\lambda)+n_\alpha(h,\lambda)),
\end{eqnarray}
where $c>0$ is a constant independent of $\alpha,$ $h$ and $\lambda.$
\end{lem}

\begin{proof}
By the triangular inequality we first get
\begin{eqnarray}\label{eq4p17}
\lefteqn{\Big|\int_{\LL^{2}_{h,\lambda}}|O_{h}^{2}(z)| \frac{\partial v_{_h}}{\partial\theta}(z)dA_{\alpha}(z)\Big|}\\\nonumber
 &&\leq \int_{\LL^{2}_{h,\lambda}}\big||O_{h}^{2}(z)|- h^{2}(\theta)\big|\big|\frac{\partial v_{_h}}{\partial\theta}(z)\big|dA_{\alpha}(z)
+ \big|\int_{\LL^{2}_{h,\lambda}}h^{2}(\theta)\frac{\partial v_{_h}}{\partial\theta}(z)dA_{\alpha}(z)\big|.
\end{eqnarray}
Now, by construction we have $|O_h(z)|\le 2h(\theta),$ for every $z\in\LL_h$,
and thus (considering the two cases $|O_h(z)|\ge \frac{1}{2}h(\theta)$ and
$|O_h(z)|\le \frac{1}{2}h(\theta)$)
\begin{equation*}
\frac{\big||O_{h}^{2}(z)|- h^{2}(\theta)\big|}{h^{2}(\theta)}\lesssim\big|\log\frac{|O_{h}(z)|}{h(\theta)}\big|,\qquad z\in\LL_h.
\end{equation*}
Note also that since $|O_h(z)|\le 2 h(\theta)$ we have
\[
 ||O_h(z)|^2-h^2(\theta)|\lesssim h^2(\theta).
\]
Incorporating  both estimates in the first integral in the right hand side of \eqref{eq4p17}, and
using \eqref{estimdv},
we get
\begin{eqnarray}\label{eqstar}
\lefteqn{\int_{\LL^{2}_{h,\lambda}}\big||O_{h}^{2}(z)|- h^{2}(\theta)\big|\big|\frac{\partial v_{_h}}{\partial\theta}(z)\big|dA_{\alpha}(z)}\nonumber\\
&&\lesssim
\int_{\LL^{2}_{h,\lambda}}h^{2}(\theta)\big(
\int_{\T_{_h}^{+}(\theta)\cup\T_{_h}^{-}(\theta)}\frac{\big|\log \frac{h(\varphi)}{h(\theta)}\big|}{|e^{i\varphi}-z|^{2}} d\varphi \big)dA_{\alpha}(z) \nonumber
 \\&&\quad  +   \int_{\LL^{2}_{h,\lambda}}h^{2}(\theta)\big|\log\frac{|O_{h}(z)|}{h(\theta)}\big|\big(
\int_{\T_{_h}(\theta)}\frac{\big|\log \frac{h(\varphi)}{h(\theta)}\big|}{|e^{i\varphi}-z|^{2}} d\varphi \big)dA_{\alpha}(z).
\end{eqnarray}
Consider the second integral in \eqref{eq4p17}. Decomposing the integral
in \eqref{log} into four pieces: $\T_h(\theta)$, $\T_h^+(\theta)$,
$\T_h^-(\theta)$ and $|e^{i\theta}-e^{i\varphi}|\le \lambda(\theta)$,
as well as $\T_h^-(\theta)$ and $|e^{i\theta}-e^{i\varphi}|\ge \lambda(\theta)$,
without applying the triangular inequality on the piece $\T_h(\theta)$, we
obtain first:
\begin{eqnarray*}
\lefteqn{
\Big|\int_{\LL^{2}_{h,\lambda}} h^{2}(\theta)\frac{\partial v_{_h}}{\partial\theta}(z)dA_{\alpha}(z)\big|
 \lesssim
\int_{\LL^{2}_{h,\lambda}}h^{2}(\theta)\big(
\int_{\T_{_h}^{+}(\theta)}\frac{\big|\log \frac{h(\varphi)}{h(\theta)}\big|}{|e^{i\varphi}-z|^{2}} d\varphi \big)dA_{\alpha}(z)} \\
 &&\qquad+\int_{\LL^{2}_{h,\lambda}}h^{2}(\theta)\big(
\int_{\T_{_h}^{-}(\theta)\atop |e^{i\theta}-e^{i\varphi}|\le \lambda(\theta)}\frac{\big|\log \frac{h(\varphi)}{h(\theta)}\big|}{|e^{i\varphi}-z|^{2}} d\varphi \big)dA_{\alpha}(z)\\
 &&\qquad +\int_{\LL^{2}_{h,\lambda}}h^{2}(\theta)\big(
\int_{\T_{_h}^{-}(\theta)\atop |e^{i\theta}-e^{i\varphi}|\ge \lambda(\theta)}\frac{\big|\log \frac{h(\varphi)}{h(\theta)}\big|}{|e^{i\varphi}-z|^{2}} d\varphi \big)dA_{\alpha}(z)\\
 &&\qquad+\Big|\int_{\LL^{2}_{h,\lambda}}h^{2}(\theta)\int_{\T_{_h}(\theta)}Q(e^{i\varphi},z)\log \frac{h(\varphi)}{h(\theta)} d\varphi dA_{\alpha}(z)
 \Big|\\
&&=\cI_{1}+\cI_{2}+\cI_{3}+\cI_4.
\end{eqnarray*}
And hence, noting that $\cI_{1}$, $\cI_{2}$ and $\cI_{3}$ also appear in the first integral
in \eqref{eqstar}, from \eqref{eq4p17} and \eqref{eqstar}, we thus get
\begin{eqnarray}\nonumber
\lefteqn{
\Big|\int_{\LL^{2}_{h,\lambda}} |O_h^{2}(z)|\frac{\partial v_{_h}}{\partial\theta}(z)dA_{\alpha}(z)\big|}
 \\
&& \lesssim
\cI_{1}+\cI_{2}+\cI_{3}+\cI_4+\int_{\LL^{2}_{h,\lambda}}h^{2}(\theta)\big|\log\frac{|O_{h}(z)|}{h(\theta)}\big|\big(
\int_{\T_{_h}(\theta)}\frac{\big|\log \frac{h(\varphi)}{h(\theta)}\big|}{|e^{i\varphi}-z|^{2}} d\varphi \big)dA_{\alpha}(z) \nonumber
\end{eqnarray}
The last integral on the right hand side will be denoted by $\cI_5$.

It is clear that
\begin{eqnarray}\label{quo0}
\log \frac{h(\varphi)}{h(\theta)}\le\frac{h(\varphi)-h(\theta)}{h(\theta)},\qquad \varphi\in\T_{_h}^+(\theta). 
\end{eqnarray}
Then, by using \eqref{circ6},
\begin{eqnarray}\label{quo00}
\cI_1\lesssim
\alpha\int_{\T_{h}}\int_{\T_h^+(\theta)}\int_0^1\frac{|h(\varphi)-h(\theta)|^2}
{|e^{i\varphi}-z|^{2}}(1-r)^{\alpha-1}drd\varphi d\theta
\lesssim  N_{\alpha}(h).
\end{eqnarray}
By the very definition of $\tilde{n}_{\alpha}(h,\lambda)$ and
using  again \eqref{circ6},
\begin{eqnarray}\label{quo1}
\cI_2\lesssim
\alpha\int_{\T_{h}}h^2(\theta)\int_{\T_h^-(\theta)\atop |e^{i\varphi}-e^{i\theta}|\le \lambda(\theta)}\log\frac{h(\theta)}{h(\varphi)}\int_0^1\frac{(1-r)^{\alpha-1}}
{|e^{i\varphi}-z|^{2}} drd\varphi d\theta
\lesssim \widetilde{n}_{\alpha}(h,\lambda).
\end{eqnarray}
Thanks to $|e^{i\varphi}-z|\geq\frac{1}{3}|e^{i\varphi}-e^{i\theta}|$, we get
\begin{eqnarray*}
\cI_3&\lesssim&\int_{\LL^{2}_{h,\lambda}}h^{2}(\theta)\big(\int_{\T_{_h}^{-}(\theta)\atop|e^{i\varphi}-e^{i\theta}|\geq\lambda(\theta)}\frac{\log \frac{h(\theta)}{h(\varphi)}}{|e^{i\varphi}-e^{i\theta}|^{2}} d\varphi \big) \alpha (1-r)^{\alpha-1}dr
 d\theta
   \\
&=&\int_{-\pi}^{\pi}h^{2}(\theta)\big(\int_{\T_{_h}^{-}(\theta)\atop|e^{i\varphi}-e^{i\theta}|\geq\lambda(\theta)}\frac{\log \frac{h(\theta)}{h(\varphi)}}{|e^{i\varphi}-e^{i\theta}|^{2}} d\varphi \big) \big(\int_{1-\rho_{h,\lambda}(\theta)\leq r\leq1}\alpha (1-r)^{\alpha-1}dr\big)
 d\theta.
\end{eqnarray*}
Now, the integral over $r$ corresponds to $\rho_{h,\lambda}^{\alpha}$ which is controlled
by $(2/a_{h,\lambda}(\theta))^{\alpha}$, and thus by definition of $a_{h,\lambda}$ and
$n_{\alpha}(h,\lambda)$ we get
\begin{eqnarray}\label{quo2}
\cI_3\lesssim n_{\alpha}(h,\lambda).
\end{eqnarray}
We now estimate the integral $\cI_{4}$ exploiting some symmetry properties of $Q$ that will
allow us to recover the quadratic difference $|h(\theta)-h(\varphi)|^2$ (see
\eqref{qu6} below). To this end,  in the equation \eqref{Q} we set $t:=\varphi-\theta,$  so that
\begin{eqnarray}\label{eleme}
Q(e^{i\varphi},z)=2r\frac{(1-r)^2-2\sin^2(t/2)(1+r^2)}{((1-r)^2+4r\sin^2(t/2))^2}
=:q(t,r).
\end{eqnarray}
In particular we remark that $q$ is even with respect to the first variable.  We define
$$\Gamma_{_h}(t):=\big\{|\theta|\leq \pi\ :e^{i\theta}\in \T_h,\ e^{i(\theta+t)}\in\T_{_h}(\theta) \big\},\qquad t\in[-\pi,\pi].$$
We note that $\theta\in\Gamma_{_h}(-t)$ if and only if $\theta-t\in\Gamma_{_h}(t).$
By a change of variables 
\begin{eqnarray}
\lefteqn{\int_{-\pi}^{0}q(t,r)\big(\int_{\Gamma_{_h}(t)}h^{2}(\theta)\log \frac{h(\theta+t)}{h(\theta)}d\theta\big) dt}\nonumber\\
\nonumber &&=\int_{0}^{\pi}q(t,r)\big(\int_{\Gamma_{_h}(-t)}h^{2}(\theta)\log \frac{h(\theta-t)}{h(\theta)}d\theta\big) dt
\\\label{ze}&&=\int_{0}^{\pi}q(t,r)\big(\int_{\Gamma_{_h}(t)}h^{2}(\theta+t)\log \frac{h(\theta)}{h(\theta+t)}d\theta\big) dt, \qquad 0<r<1.
\end{eqnarray}
Therefore (note that the change of variables $(\varphi,\theta)= (t+\theta,\theta)$ is harmless),
\begin{eqnarray}\nonumber
\lefteqn{\int_{\D}h^{2}(\theta) \Big(\int_{\T_{_h}(\theta)}Q(e^{i\varphi},z)\log \frac{h(\varphi)}{h(\theta)} d\varphi \Big)dA_{\alpha}(z)}
\\\nonumber&&=\alpha\int_{0}^{1}\Big(\int_{-\pi}^{\pi}h^{2}(\theta) \big(\int_{
|t|\leq\pi\atop e^{i(\theta+t)}\in\T_{_h}(\theta)}q(t,r)\log \frac{h(\theta+t)}{h(\theta)}dt\big)d\theta\Big)(1-r)^{\alpha-1}dr
\\\nonumber&&=\alpha\int_{0}^{1}\Big(\int_{-\pi}^{\pi}q(t,r)\big(\int_{\Gamma_{_h}(t)}h^{2}(\theta)\log \frac{h(\theta+t)}{h(\theta)}d\theta\big) dt\Big)(1-r)^{\alpha-1}dr
\\\nonumber&&=\alpha\int_{0}^{1}\left\{\int_{0}^{\pi}q(t,r)
 \times \left[\int_{\Gamma_{_h}(t)}h^{2}(\theta)\log \frac{h(\theta+t)}{h(\theta)}d\theta\right.\right.
\\\nonumber&&\qquad \left.\left.+\int_{\Gamma_{_h}(t)}h^{2}(\theta+t)\log \frac{h(\theta)}{h(\theta+t)}d\theta\right] dt\right\}(1-r)^{\alpha-1}dr
\\\label{zero}&&=-\alpha\int_{0}^{1}\Big(\int_{0}^{\pi}q(t,r)\big(\int_{\Gamma_{_h}(t)}\mathcal{P}_h(\theta, t)d\theta\big) dt\Big)(1-r)^{\alpha-1}dr,
\end{eqnarray}
where we have used \eqref{ze} and
$$\mathcal{P}_h(\theta, t):=\big(h^{2}(\theta+t)-h^{2}(\theta)\big)\big(\log h(\theta+t)- \log h(\theta)\big).$$
Since for $\theta\in \Gamma_h(t)$ we have $e^{i(\theta+t)}\in\T_{_h}(\theta) $, i.e.\ $h(\theta+t)
\stackrel{2}{\asymp} h(\theta)$, it is clear (see e.g.\ \eqref{logequiv}) that we get
the desired quadratic difference
\begin{eqnarray}\label{qu6}
0\leq \mathcal{P}_h(\theta, t)\leq 4\big(h(\theta+t)- h(\theta)\big)^{2} ,\qquad t\in[-\pi,\pi]\text{ and }\theta\in\Gamma_{_h}(t).
\end{eqnarray}
Observe that the function we integrate over $\LL_{h,\lambda}^2$ is not positive, so that we 
cannot just replace this domain by $\D$.
Still, writing $\LL^{2}_{h,\lambda}=\D\setminus(\K_h\cup\LL^{1}_{h,\lambda})$,
the triangular inequality obviously yields 
\begin{eqnarray}\nonumber
\cI_4
&\leq& \Big|\int_{\D}h^{2}(\theta)\Big(\int_{\T_{_h}(\theta)}Q(e^{i\varphi},z)\log \frac{h(\varphi)}{h(\theta)} d\varphi \Big)dA_{\alpha}(z)\Big|
\\\nonumber&&+
\Big|\int_{\LL^{1}_{h,\lambda}}h^{2}(\theta)\Big(\int_{\T_{_h}(\theta)}Q(e^{i\varphi},z)\log \frac{h(\varphi)}{h(\theta)} d\varphi \Big)dA_{\alpha}(z)\Big|
\\\label{h1}&&+
\Big|\int_{\K_{h}}h^{2}(\theta)\Big(\int_{\T_{_h}(\theta)}Q(e^{i\varphi},z)\log \frac{h(\varphi)}{h(\theta)} d\varphi \Big)dA_{\alpha}(z)\Big|.
\end{eqnarray}
We can now use the triangular inequality in the integral over $\D$.
From \eqref{qu2}, $|q(t,r)|=|Q(e^{i\varphi},re^{i\theta})|\le 2/|e^{i\varphi}-z|^2=2/|e^{i(t+\theta)}
-re^{i\theta}|^2$, 
and by \eqref{circ6}, $\int_0^1{\frac{\dst (1-r)^{\alpha-1}}
{\dst|e^{i(t+\theta)}-e^{i\theta}|^2}}dr\lesssim 1/|e^{i(t+\theta)}-e^{i\theta}|^{2-\alpha}$. Hence,
from \eqref{zero} et \eqref{qu6}, we deduce that
\begin{eqnarray}\label{h2}
\Big|\int_{\D}h^{2}(\theta)\Big(\int_{\T_{_h}(\theta)}Q(e^{i\varphi},z)\log \frac{h(\varphi)}{h(\theta)} d\varphi \Big)dA_{\alpha}(z)\Big|
\lesssim N_\alpha(h).
\end{eqnarray}
(Without our symmetry argument, the triangular inequality together with the estimates 
\eqref{qu2} and \eqref{circ6} would only have given the linear
difference which is not enough.)

Next, since on $\T_h(\theta)$, $|\log(h(\varphi)/h(\theta))|\le \log 2$ and 
$|Q(e^{i\varphi},z)|\lesssim 1/|e^{i\varphi}-z|^2$,  using the standard integration of 
the Poisson kernel
\begin{eqnarray}\label{standestim}
 \int_{\T}\frac{1}{|e^{i\varphi}-re^{i\theta}|^2}d\varphi
=\frac{2\pi}{1-r^2},
\end{eqnarray}
we get
\begin{eqnarray*}\nonumber
\lefteqn{\Big|\int_{\LL^{1}_{h,\lambda}}h^{2}(\theta)\Big(\int_{\T_{_h}(\theta)}Q(e^{i\varphi},z)\log \frac{h(\varphi)}{h(\theta)} d\varphi \Big)dA_{\alpha}(z)\Big|}
\\\nonumber&&\lesssim\int_{\LL^{1}_{h,\lambda}}h^{2}(\theta)(1-r)^{-1}dA_{\alpha}(z)
\\\nonumber
 &&= \frac{\alpha}{1-\alpha}\int_{-\pi}^{\pi}h^{2}(\theta)\rho_{h,\lambda}^{\alpha-1}(\theta) d\theta
-\frac{\alpha}{1-\alpha}\int_{-\pi}^{\pi}h^{2}(\theta)d\theta,
\\&&\leq \frac{\alpha}{1-\alpha}\int_{\rho_{h,\lambda}(\theta)<1}h^{2}(\theta)\rho_{h,\lambda}^{\alpha-1}(\theta) d\theta.
\end{eqnarray*}
Hence, as in the proof of Lemma \ref{l1h}, we obtain
\begin{eqnarray}\nonumber
\lefteqn{\Big|\int_{\LL^{1}_{h,\lambda}}h^{2}(\theta)\Big(\int_{\T_{_h}(\theta)}Q(e^{i\varphi},z)\log \frac{h(\varphi)}{h(\theta)} d\varphi \Big)dA_{\alpha}(z)\Big|}
\\\label{h3}&&\lesssim\frac{\alpha}{1-\alpha}(\widetilde{n}_\alpha(h,\lambda)+n_\alpha(h,\lambda)).
\end{eqnarray}
For the integral over $\K_h$ we start with the same argument as above (since in the inner
integral we indeed integrate over $\T_h(\theta)$) to get
\begin{eqnarray*}\nonumber
\Big|\int_{\K_{h}}h^{2}(\theta)\Big(\int_{\T_{_h}(\theta)}Q(e^{i\varphi},z)\log \frac{h(\varphi)}{h(\theta)} d\varphi \Big)dA_{\alpha}(z)\Big|
\lesssim \int_{\K_{h}}h^{2}(\theta)(1-r)^{-1}dA_{\alpha}(z).
\end{eqnarray*}
By the very definition of  $\K_{h}$ and the inequalities \eqref{circ}  --- implying in particular
$h(\theta)\lesssim \sup_{w\in \D(re^{i\theta})}|O_h(w)|\lesssim \mathcal{H}(z)$ --- and \eqref{circ6},
\begin{eqnarray*}\nonumber
\int_{\K_{h}}h^{2}(\theta)(1-r)^{-1}dA_{\alpha}(z)
&\lesssim& \int_{\K_{h}}\frac{\cH^2(z)}{1-r}dA_{\alpha}(z)
\\&\lesssim& \int_{\K_{h}}\int_{-\pi}^{\pi}\frac{\big|h(\varphi)-h(\theta)\big|^{2}}{|e^{i\varphi}-z|^{2}} d\varphi dA_{\alpha}(z)
\\&\lesssim& N_\alpha(h).
\end{eqnarray*}
Thus
\begin{eqnarray}\label{h4}
\Big|\int_{\K_{h}}h^{2}(\theta)\Big(\int_{\T_{_h}(\theta)}Q(e^{i\varphi},z)\log \frac{h(\varphi)}{h(\theta)} d\varphi \Big)dA_{\alpha}(z)\Big|
\lesssim N_\alpha(h).
\end{eqnarray}

By combining \eqref{h1}, \eqref{h2}, \eqref{h3} and \eqref{h4}
\begin{eqnarray}\nonumber
\cI_4\lesssim N_\alpha(h)+ \frac{\alpha}{1-\alpha}(\widetilde{n}_\alpha(h,\lambda)+n_\alpha(h,\lambda)).
\end{eqnarray}

It remains to estimate $\cI_5$. Using first the very definition of the outer function $O_h$
(so that $\log |O_h|$ is just the Poisson extension of $\log h$ at $z=re^{i\theta}$)
and then rearranging terms,
\begin{eqnarray}\label{eq4.29}\nonumber
\cI_{5}
&\leq& \int_{\LL^{2}_{h,\lambda}}h^{2}(\theta)\left(
\frac{1}{2\pi}\int_{-\pi}^{\pi}\frac{1-r^{2}}{|e^{i\varphi}-z|^{2}}\big|\log \frac{h(\varphi)}{h(\theta)}\big| d\varphi \right)\times
\\\nonumber&&\hspace{5cm}\times\left(
\int_{\T_{_h}(\theta)}\frac{\big|\log \frac{h(\varphi)}{h(\theta)}\big|}{|e^{i\varphi}-z|^{2}} d\varphi \right)dA_{\alpha}(z)
\\&=& 2\pi\int_{\LL^{2}_{h,\lambda}}h^{2}(\theta)\left(
\frac{1}{2\pi}\int_{\T_{_h}(\theta)}\frac{\big|\log \frac{h(\varphi)}{h(\theta)}\big|}{|e^{i\varphi}-z|^{2}} d\varphi \right)^{2}(1-r^{2})dA_{\alpha}(z)
\\\nonumber&&+ \int_{\LL^{2}_{h,\lambda}}h^{2}(\theta)\left(
\int_{\T_{_h}^{+}(\theta)\cup\T_{_h}^{-}(\theta)}\frac{\big|\log \frac{h(\varphi)}{h(\theta)}\big|}{|e^{i\varphi}-z|^{2}} d\varphi \right)\times
\\\nonumber &&\hspace{4cm}\times\big(
\frac{1}{2\pi}\int_{\T_{_h}(\theta)}\frac{1-r^2}{|e^{i\varphi}-z|^{2}}\big|\log \frac{h(\varphi)}{h(\theta)}\big| d\varphi \big)dA_{\alpha}(z).
\end{eqnarray}
Consider the first integral. Recall that on $\T_h(\theta)$ we have
\[
 |\log\frac{h(\theta)}{h(\varphi)}|\asymp\Big|\frac{h(\theta)-h(\varphi)}{h(\theta)}\Big|,
\]
so that by Jensen's inequality,
\begin{eqnarray*}
\lefteqn{h(\theta)^2\left(
\frac{1}{2\pi}\int_{\T_{_h}(\theta)}\frac{\big|\log \frac{h(\varphi)}{h(\theta)}\big|}{|e^{i\varphi}-z|^{2}} d\varphi \right)^{2}}\\
&&\asymp \left(
\frac{1}{2\pi}\int_{\T_{_h}(\theta)}\frac{1-r^2}{|e^{i\varphi}-z|^{2}}
 |{h(\theta)-h(\varphi)}|d\varphi \right)^{2}\times\frac{1}{(1-r^2)^2}\\
 &&\le\frac{1}{2\pi}\int_{\T_{_h}(\theta)}\frac{1-r^2}{|e^{i\varphi}-z|^{2}}
 |{h(\theta)}-{h(\varphi)}|^2d\varphi \times\frac{1}{(1-r^2)^2}.
\end{eqnarray*}
Hence using again \eqref{circ6}
\begin{eqnarray*}
\lefteqn{2\pi\int_{\LL^{2}_{h,\lambda}}h^{2}(\theta)\left(
\frac{1}{2\pi}\int_{\T_{_h}(\theta)}\frac{\big|\log \frac{h(\varphi)}{h(\theta)}\big|}{|e^{i\varphi}-z|^{2}} d\varphi \right)^{2}(1-r^{2})dA_{\alpha}(z)}\\
 &&\lesssim \int_{\LL^{2}_{h,\lambda}}\int_{\T_{_h}(\theta)}\frac{\big|h(\varphi)-h(\theta)\big|^2}{|e^{i\varphi}-z|^{2}}dA_{\alpha}(z) d\varphi\lesssim N_{\alpha}(h).
\end{eqnarray*}
Consider the second term in \eqref{eq4.29}. Since on $\T_h(\theta)$, the expression
$\log\frac{h(\varphi)}{h(\theta)}$ is bounded, the last factor is bounded by
a constant, so that this term is controlled by
$\log 2 \times  (\cI_1+\cI_2+\cI_3)$.
\end{proof}

As a conclusion, the desired  estimate \eqref{sufficient} follows from Lemmas  \ref{harmonic}, \ref{l1h} and \ref{l2h}.

\section{\bf The necessity}\label{sect4}

In this section we show that if  $O_{h}\in\cD_{\alpha},$ then
\begin{equation}\label{necessaire}
\|O_h\|_{\cD_{\alpha}} 
\gtrsim N_\alpha(h)+ n_\alpha(h) +(1-\alpha)^2\widetilde{n}_\alpha(h).
\end{equation}
Note that $\|O_h\|_{\cD_{\alpha}} \asymp \|h\|_2^2+\cD_{\alpha}(O_h)$ 
and that we have already observed  (see \eqref{estimNalpha}) that
\[
 \cD_{\alpha}(O_h)\gtrsim N_{\alpha}(h).
\]


We start with the following Lemma.

\begin{lem}\label{necess0}
We have
\begin{equation*}
\|O_h\|^{2}_{\cD_{\alpha}}\gtrsim n_\alpha(h),
\end{equation*}
independently of both $\alpha$ and  $h.$
\end{lem}

\begin{proof}
From Lemma \ref{lem2}, we know that $\mu_{_h}\in\Lambda.$ Then, setting $a_h=a_{h,\mu_{_h}}$,
by definition  of $\mu_{_h}$ (see \eqref{nana1}),
\begin{equation}\label{touta1}
\mu_{_h}(\theta) a_{_{h}}(\theta)\leq2, \quad \text{a.e.\ on }\T. 
\end{equation}
Thus, by Lemma \ref{dauglas},
\begin{eqnarray}\label{equa1}
\cD_{\alpha}(O_h)\gtrsim \int_{\mu_{_h}(\theta)<1}h^{2}(\theta)\mu^{\alpha-1}_{h}(\theta)d\theta\gtrsim\int_{\mu_{_h}(\theta)<1}h^{2}(\theta)a_{_{h}}^{1-\alpha}(\theta)d\theta.
\end{eqnarray}
Note that when $\mu_{_h}(\theta)=1$, again by \eqref{touta1} $a_h(\theta)\le 2$, and so
\begin{eqnarray}\label{equa2}
\|h\|^{2}_{2}\gtrsim\int_{\mu_{_h}(\theta)=1}h^{2}(\theta)a_{_{h}}^{1-\alpha}(\theta)d\theta.
\end{eqnarray}
The proof is completed by adding the inequalities \eqref{equa1} and \eqref{equa2} together.
\end{proof}

The most difficult part of the proof of the necessity is the control of $\tilde{n}_{\alpha}(h)$.
We set $$\M_{h}:=\{z\in\D_{_h}\ :\ r\geq 1-\mu_{_h}(\theta)\}.$$
As it turns out it is integration on $\M_h$ which will yield the desired control.
\\

We start with the following simple auxiliary lemma which is certainly well known,  but for which
we produce a proof
here for the convenience of the reader thereby exhibiting the right control of the
constants.

\begin{lem}\label{hssab}
We fix two real numbers $0< \mu \leq 1$ and $0<u\leq 2.$ Then
\begin{eqnarray*}
\frac{1}{1-\alpha}\int_{1-\mu}^{1}\frac{ru^2-(1-r)^2}{((1-r)^2+ru^2)^2}(1-r)^{\alpha-1}rdr
\geq cu^{\alpha-2}, \qquad \hbox{ if}\quad 0< u\leq \mu,
\end{eqnarray*}
where $c>0$ is a constant independent of $\alpha.$
\end{lem}

Notice that while the function we integrate is not necessarily positive on the integration
domain, the integral itself will be positive.

\begin{proof}
We have
\begin{eqnarray}\nonumber
&&\int_{1-\mu}^{1}\frac{ru^2-(1-r)^2}{((1-r)^2+ru^2)^2}(1-r)^{\alpha-1}rdr
\\\nonumber&=&-\frac{\partial}{\partial u} \Big(\int_{1-\mu}^{1}\frac{u}{(1-r)^2+ru^2}(1-r)^{\alpha-1}rdr\Big)
\\\nonumber&=&-\frac{\partial}{\partial u} \Big(u^{\alpha-1}\int_{0}^{\mu/u}\frac{s^{\alpha-1}}{1-us+s^2}(1-us)ds\Big),
 \qquad \text{where }1-r=su,
\\\nonumber&=& (1-\alpha)u^{\alpha-2}\int_{0}^{\mu/u}\frac{s^{\alpha-1}}{1-us+s^2}(1-us)ds
+\frac{\mu^{\alpha}(1-\mu)}{u^{2}(1-\mu)+\mu^{2}}
\\\nonumber&&+u^{\alpha-1}\int_{0}^{\mu/u}\frac{s^{2+\alpha}}{(1-us+s^{2})^{2}}ds
\\\nonumber&\geq&(1-\alpha)u^{\alpha-2}\int_{0}^{\mu/u}\frac{s^{\alpha-1}}{1-us+s^2}(1-us)ds
\\\label{lema}&\geq&
\frac{1-\alpha}{4}u^{\alpha-2}\int_{0}^{1/4}s^{\alpha-1}ds, \qquad \hbox{ if}\quad u\leq \mu,
\end{eqnarray}
from where we deduce the assertion of
Lemma \ref{hssab}.
\end{proof}

Recall from \eqref{motassa2} that
$\cD_{\alpha}(f)\asymp \frac{1}{2\pi}\int_{\D} |f(z)|^2\frac{\partial v}{\partial \theta}(z) dA_{\alpha}(z)$,
and from
\eqref{log} that $\frac{\partial v_{_h}}{\partial\theta}(z)
=\frac{1}{2\pi}\int_{-\pi}^{\pi}Q(e^{i\varphi},z)\log \frac{h(\varphi)}{h(\theta)} d\varphi$.
The next two lemmas allow to obtain the control of $\tilde{n}_{\alpha}(h)$.

\begin{lem}\label{necess3}
We have
\begin{eqnarray*}
\int_{\M_{h}}h^{2}(\theta) \big(\int_{\T^{-}_{_h}(\theta)\atop|e^{i\varphi}-e^{i\theta}|\leq \mu_{_h}(\theta)}Q(e^{i\varphi},z)\log\frac{h(\varphi)}{h(\theta)} d\varphi \big)dA_{\alpha}(z)
\geq c(1-\alpha)\widetilde{n}_{\alpha}(h),
\end{eqnarray*}
where $c>0$ is a constant independent of both $\alpha$ and $h.$ 
\end{lem}

\begin{proof}
We have
\begin{eqnarray}\label{f3}
&&\int_{\M_{h}}h^{2}(\theta) \big(\int_{\T^{-}_{_h}(\theta)\atop|e^{i\varphi}-e^{i\theta}|\leq \mu_{_h}(\theta)}Q(e^{i\varphi},z)\log\frac{h(\varphi)}{h(\theta)} d\varphi \big)dA_{\alpha}(z)
\\\nonumber&=&\int_{-\pi}^{\pi}h^{2}(\theta)
\Big(\int_{\T^{-}_{_h}(\theta)\atop|e^{i\varphi}-e^{i\theta}|\leq \mu_{_h}(\theta)}\log\frac{h(\theta)}{h(\varphi)} \big(-\alpha\int_{1-\mu_{_h}(\theta)}^{1}Q(e^{i\varphi},z)(1-r)^{\alpha-1}dr\big)d\varphi \Big)d\theta.
\end{eqnarray}
Let $e^{i\theta}\in\T_{_h}$ be a point such that $\mu_{_h}(\theta)=:\mu>0.$  For a fixed number  $0<t=\varphi-\theta\leq\pi$ we set  $u:=2\sin(t/2).$
With \eqref{eleme} in mind,
\begin{eqnarray*}
\lefteqn{-\alpha\int_{1-\mu}^{1} q(t,r)(1-r)^{\alpha-1}dr}
\\&&=\alpha\int_{1-\mu}^{1} 2\frac{2\sin^2(t/2)(1+r^2)-(1-r)^2}{((1-r)^2+4r\sin^2(t/2))^2}(1-r)^{\alpha-1}rdr
\\&&\geq 2\alpha\int_{1-\mu}^{1}\frac{ru^2-(1-r)^2}{((1-r)^2+ru^2)^2}(1-r)^{\alpha-1}rdr,
\end{eqnarray*}
where in the last inequality we have used $1+r^2\geq 2r$ and the fact that the denominator
is positive (again, the function we integrate is not necessarily positive on the whole
integration interval).

Now,
an easy computation gives
$|u|=|2\sin(t/2)|=|e^{i\varphi}-e^{i\theta}|$ which is supposed to be bounded
by $\mu$ in \eqref{f3}, so that by Lemma \ref{hssab}, we get
\begin{eqnarray}\label{f4}
  -\alpha\int_{1-\mu}^{1} q(t,r)(1-r)^{\alpha-1}dr
\gtrsim (1-\alpha)u^{\alpha-2}=\frac{1-\alpha}{|e^{i\varphi}-e^{i\theta}|^{2-\alpha}}.
\end{eqnarray}
The estimate in Lemma \ref{necess3} follows from \eqref{f3}, \eqref{f4} and the very definition of
$\tilde{n}_{\alpha}(h).$ 
\end{proof}

The next lemma connects the previous estimate with $\cD_{\alpha}(O_h)$.

\begin{lem}\label{necess2}
We have
\begin{eqnarray}\nonumber
{\mathcal{D}_{\alpha}(O_h)}+\|O_h\|_2^2
\gtrsim \int_{\M_{h}}h^{2}(\theta) \big(\int_{\T^{-}_{_h}(\theta)\atop|e^{i\varphi}-e^{i\theta}|\leq \mu_{_h}(\theta)}Q(e^{i\varphi},z)\log\frac{h(\varphi)}{h(\theta)} d\varphi \big)dA_{\alpha}(z),
\end{eqnarray}
independently of both $\alpha$ and $h.$ 
\end{lem}

\begin{proof}
Recall that
\[
 \frac{\partial v_h}{\partial \theta}=\int_{-\pi}^{\pi}
 Q(e^{i\varphi},z)\log \frac{h(\varphi)}{h(\theta)}d\varphi,
\]
and hence
\begin{eqnarray}\nonumber
\lefteqn{\int_{\M_{h}}h^{2}(\theta) \big(\int_{\T^{-}_{_h}(\theta)\atop|e^{i\varphi}-e^{i\theta}|\leq \mu_{_h}(\theta)}Q(e^{i\varphi},z)\log\frac{h(\varphi)}{h(\theta)} d\varphi \big)dA_{\alpha}(z)}
\\\nonumber&&=\int_{\M_{h}}h^{2}(\theta) \frac{\partial v_{_h}}{\partial\theta}(z)dA_{\alpha}(z)
\\\nonumber&&\quad-\int_{\M_h }h^{2}(\theta) \big(\int_{\T^{-}_{_h}(\theta)\atop|e^{i\varphi}-e^{i\theta}|\geq \mu_{_h}(\theta)}Q(e^{i\varphi},z)\log\frac{h(\varphi)}{h(\theta)} d\varphi \big)dA_{\alpha}(z)
\\\nonumber&&\quad- \int_{\M_{h}}h^{2}(\theta) \big(\int_{\T^{+}_{_h}(\theta)}Q(e^{i\varphi},z)\log\frac{h(\varphi)}{h(\theta)} d\varphi \big)dA_{\alpha}(z)
\\\label{g1}&&\quad- \int_{\M_{h}}h^{2}(\theta) \big(\int_{\T_{_h}(\theta)}Q(e^{i\varphi},z)\log\frac{h(\varphi)}{h(\theta)} d\varphi \big)dA_{\alpha}(z).
\end{eqnarray}
We will now estimate the 4 integrals appearing above.

The proof of the first estimate
\begin{eqnarray}\label{g2}
 \left|\int_{\M_{h}}h^{2}(\theta) \frac{\partial v_{_h}}{\partial\theta}(z)dA_{\alpha}(z)\right|
 \lesssim  \frac{1}{1-\alpha}\mathcal{D}_{\alpha}(O_h)
\end{eqnarray}
is lengthier, and we prefer to postpone it to the end of this section (see Lemma \ref{necess1}).

Next, from  \eqref{qu2} we get
\begin{eqnarray*}
&&\Big|\int_{\M_{h}}h^{2}(\theta) \big(\int_{\T^{-}_{_h}(\theta)\atop|e^{i\varphi}-e^{i\theta}|\geq \mu_{_h}(\theta)}Q(e^{i\varphi},z)\log\frac{h(\varphi)}{h(\theta)} d\varphi \big)dA_{\alpha}(z)\Big|
\\&\lesssim&
\int_{\M_{h}}h^{2}(\theta) \big(\int_{\T^{-}_{_h}(\theta)\atop|e^{i\varphi}-e^{i\theta}|\geq \mu_{_h}(\theta)}
\frac{\log\frac{h(\theta)}{h(\varphi)}}{|e^{i\varphi}-e^{i\theta}|^2} d\varphi \big)dA_{\alpha}(z)
\\&=&\int_{-\pi}^{\pi}h^{2}(\theta) a_{_h}(\theta)\big(\int_{1-\mu_{_h}(\theta)}^1\alpha(1-r)^{\alpha-1} dr\big) d\theta.
\end{eqnarray*}
Obviously $\int_{1-\mu_{_h}(\theta)}^1\alpha(1-r)^{\alpha-1} dr=\mu_{_h}(\theta)^{\alpha}$,
and by definition $a_h(\theta)\le 2/\mu_{_h}(\theta)$. Hence,
with Lemma \ref{dauglas}, and in particular \eqref{Lem3.3bis},
\begin{eqnarray}\nonumber
\lefteqn{\Big|\int_{\M_{h}}h^{2}(\theta) \big(\int_{\T^{-}_{_h}(\theta)\atop|e^{i\varphi}-e^{i\theta}|\geq \mu_{_h}(\theta)}Q(e^{i\varphi},z)\log\frac{h(\varphi)}{h(\theta)} d\varphi \big)dA_{\alpha}(z)\Big|}\\
&&\label{g3}\lesssim \int_{-\pi}^{\pi}h^{2}(\theta) \mu^{\alpha-1}_{_h}(\theta)d\theta
 \lesssim \|h\|^{2}_{2}+\mathcal{D}_{\alpha}(O_h).
\end{eqnarray}

Consider the integral on $\T_h^+(\theta)$. Again using \eqref{qu2} and \eqref{circ6}
\[
 \alpha\int_{1-\mu}^1 Q(e^{i\varphi},z)(1-r)^{\alpha - 1}dr\lesssim
  \frac{1}{|e^{i\varphi}-e^{i\theta}|^{2-\alpha}}.
\]
Since on $\T_h^+(\theta)$ we have $h^2(\theta)\log \frac{h(\varphi)}{h(\theta)}
\le h(\theta)h(\varphi)\lesssim (h(\varphi)-h(\theta))^2$, we get
\begin{eqnarray}\label{g4}
\int_{\M_{h}}h^{2}(\theta) \big(\int_{\T^{+}_{_h}(\theta)}Q(e^{i\varphi},z)\log\frac{h(\varphi)}{h(\theta)} d\varphi \big)dA_{\alpha}(z)
\lesssim N_\alpha(h)\lesssim \mathcal{D}_{\alpha}(O_h).
\end{eqnarray}

For the last integral, we have
\begin{eqnarray}\nonumber
\lefteqn{\int_{\M_{h}}h^{2}(\theta) \big(\int_{\T_{_h}(\theta)}Q(e^{i\varphi},z)\log\frac{h(\varphi)}{h(\theta)} d\varphi \big)dA_{\alpha}(z)}
\\\nonumber&&=\int_{\D}h^{2}(\theta) \big(\int_{\T_{_h}(\theta)}Q(e^{i\varphi},z)\log\frac{h(\varphi)}{h(\theta)} d\varphi \big)dA_{\alpha}(z)
\\\nonumber&&\quad- \int_{\D\setminus \M_{h}}h^{2}(\theta) \big(\int_{\T_{_h}(\theta)}Q(e^{i\varphi},z)\log\frac{h(\varphi)}{h(\theta)} d\varphi \big)dA_{\alpha}(z).
\end{eqnarray}
As in \eqref{h2}
\begin{eqnarray}\nonumber
\Big|\int_{\D}h^{2}(\theta) \big(\int_{\T_{_h}(\theta)}Q(e^{i\varphi},z)\log\frac{h(\varphi)}{h(\theta)} d\varphi \big)dA_{\alpha}(z)\Big|
\lesssim N_{\alpha}(h).
\end{eqnarray}
Since on $\T_h(\theta)$, $|\log(h(\varphi)/h(\theta))|\le \log 2$, and
with \eqref{qu2} and \eqref{standestim} in mind,
$$\Big|\int_{\T_{_h}(\theta)}Q(e^{i\varphi},z)\log\frac{h(\varphi)}{h(\theta)} d\varphi \Big|\lesssim \frac{1}{1-r}.$$
Then, by using Lemma \ref{dauglas},
\begin{eqnarray}\nonumber
\lefteqn{\Big|\int_{\D\setminus \M_{h}}h^{2}(\theta) \big(\int_{\T_{_h}(\theta)}Q(e^{i\varphi},z)\log\frac{h(\varphi)}{h(\theta)} d\varphi \big)dA_{\alpha}(z)\Big|}
\\\nonumber&&\lesssim
\Big|\int_{\D\setminus \M_{h}}h^{2}(\theta) (1-r)^{-1}dA_{\alpha}(z)\Big|
=\alpha\int_{-\pi}^{\pi}h^2(\theta)\int_{r\le 1-\mu_{_h}(\theta)}\frac{1}{(1-r)^{2-\alpha}}drd\theta
\\\nonumber&&\lesssim \frac{\alpha}{1-\alpha}\int_{\mu_{_h}(\theta)<1}h^{2}(\theta)\mu^{\alpha-1}_{_h}(\theta)d\theta
\\\nonumber&&\lesssim \frac{\alpha}{1-\alpha}\cD_{\alpha}(O_h).
\end{eqnarray}
Therefore
\begin{eqnarray}\label{g5}
\Big|\int_{\M_{h}}h^{2}(\theta) \big(\int_{\T_{_h}(\theta)}Q(e^{i\varphi},z)\log\frac{h(\varphi)}{h(\theta)} d\varphi \big)dA_{\alpha}(z)\Big|
\lesssim\frac{1}{1-\alpha}\cD_{\alpha}(O_h).
\end{eqnarray}
Taking \eqref{g2} for granted (see Lemma \ref{necess1} below),
the desired result follows from this estimate as well as from the estimates \eqref{g1}, 
\eqref{g3}, \eqref{g4} and \eqref{g5}.
\end{proof}

To finish the proof of the necessary condition of Theorem \ref{ra},
i.e. \eqref{necessaire}, it suffices to combine
Lemmas  \ref{necess0}, \ref{necess3} and \ref{necess2}.\\

We finish this section with the proof of \eqref{g2}:

\begin{lem}\label{necess1}
We have
\begin{eqnarray*}
\cD_{\alpha}(O_h)\gtrsim \Big|\int_{\M_{h}}h^{2}(\theta)  \frac{\partial v_{_h}}{\partial\theta}(z)dA_{\alpha}(z)\Big|,
\end{eqnarray*}
independently of both $\alpha$ and  $h.$
\end{lem}

\begin{proof}
Using \eqref{veb6} and Lemma \ref{harmonic},
\begin{eqnarray}\label{nsit1}
\int_{\M_{h}\cap\K_h}h^{2}(\theta)  \big| \frac{\partial v_{_h}}{\partial\theta}(z)\big|dA_{\alpha}(z)
&\leq& e^{82}\int_{\M_{h}\cap\K_h}  \big| O_{h}^{2}(z)\frac{\partial v_{_h}}{\partial\theta}(z)\big|dA_{\alpha}(z)
\nonumber\\&\lesssim& N_{\alpha}(h)\lesssim \cD_{\alpha}(O_h).
\end{eqnarray}
With \eqref{touta1} in mind, we observe that $\M_h\setminus\K_h=\LL^{2}_{h},$ where $\LL^{2}_{h}:=\LL^{2}_{h,\mu_{_h}}$.
Since $h(\theta)^2=|O_h^2(\theta)|\le |O_h^2(\theta)-O_h^2(z)|+|O_h(z)^2|$,
we have
\begin{eqnarray}\nonumber
&&\Big|\int_{\LL^{2}_{h}}h^{2}(\theta) \frac{\partial v_{_h}}{\partial\theta}(z)dA_{\alpha}(z)\Big|
\\\nonumber&\leq& \int_{\LL^{2}_{h}}\big|O^{2}_{h}(z)-O^{2}_{h}(\theta)\big|\big| \frac{\partial v_{_h}}{\partial\theta}(z)\big|dA_{\alpha}(z)
+ \Big|\int_{\LL^{2}_{h}}|O_{h}^{2}(z)| \frac{\partial v_{_h}}{\partial\theta}(z)dA_{\alpha}(z)\Big|
\\\label{ss11}&=:& \cI_{1}+\cI_{2}.
\end{eqnarray}
The following two facts are well known.
\begin{eqnarray}
\label{ss33}
xy\leq\frac{1}{2}(x^2+y^2), \qquad x,y>0,
\end{eqnarray}
\begin{eqnarray} \label{ss331}
(x+y)^2\leq2(x^2+y^2), \qquad x,y>0.
\end{eqnarray}

In particular, using first that $|O_h(z)|\le 2h(\theta)$ for $z\in \LL_h^2$, and then \eqref{ss33}
(setting $y=|O_h(z)-O_h(\theta)|/\sqrt{1-r}$),
yields
\begin{eqnarray}\nonumber
\cI_{1}
&\leq& 3\int_{\LL^{2}_{h}}h(\theta)\big|O_{h}(z)-O_{h}(\theta)\big|\big| \frac{\partial v_{_h}}{\partial\theta}(z)\big|dA_{\alpha}(z) 
\\\nonumber&\leq&\frac{3}{2}\int_{\LL^{2}_{h}}\frac{\big|O_{h}(z)-O_{h}(\theta)\big|^{2}}{1-r}dA_{\alpha}(z)
+\frac{3}{2}\int_{\LL^{2}_{h}}h^{2}(\theta)\big| \frac{\partial v_{_h}}{\partial\theta}(z)\big|^{2}(1-r)dA_{\alpha}(z).
\end{eqnarray}

Now, by \eqref{estimdv},
\begin{eqnarray}\nonumber
\big| \frac{\partial v_{_h}}{\partial\theta}(z)\big|
&\leq&\frac{1}{\pi}\int_{-\pi}^{\pi}\frac{\big|\log\frac{h(\varphi)}{h(\theta)}\big|}{|e^{i\varphi}-z|^{2}} d\varphi\nonumber\\
&\le& \Big|\frac{1}{\pi}\int_{h(\varphi)\le h(\theta)}\frac{\log\frac{h(\varphi)}{h(\theta)}}{|e^{i\varphi}-z|^{2}} d\varphi\Big|
+\frac{1}{\pi}\int_{h(\varphi)\ge h(\theta)}\frac{\log\frac{h(\varphi)}{h(\theta)}}{|e^{i\varphi}-z|^{2}} d\varphi\nonumber
\\\label{ss3}&\leq&2\frac{\big|\log\frac{|O_{h}(z)|}{h(\theta)}\big|}{1-r}
+\frac{2}{\pi}\int_{h(\varphi)\geq h(\theta)}\frac{\log\frac{h(\varphi)}{h(\theta)}}{|e^{i\varphi}-z|^{2}} d\varphi
,\qquad z\in\D_{_h}.
\end{eqnarray}

Applying this and
\eqref{ss331} to the sum on the right hand side in \eqref{ss3}
yields
\begin{eqnarray}\nonumber
\cI_{1}
&\lesssim&\int_{\LL^{2}_{h}}\frac{\big|O_{h}(z)-O_{h}(\theta)\big|^{2}}{1-r}dA_{\alpha}(z)
+\int_{\LL^{2}_{h}}h^{2}(\theta)\frac{\big|\log\frac{|O_{h}(z)|}{h(\theta)}\big|^{2}}{1-r}dA_{\alpha}(z)
\\\label{bord0}&&+  \int_{\LL^{2}_{h}}h^{2}(\theta)\big(\int_{ h(\varphi)\geq h(\theta)}
\frac{\big|\log\frac{h(\varphi)}{h(\theta)}\big|}{|e^{i\varphi}-z|^{2}} d\varphi \big)^{2}|(1-r)dA_{\alpha}(z).
\end{eqnarray}
In the above, the second integral is controlled by the first one since for almost
all points $z\in \LL^2_h$, we have 
from \eqref{veb6},
\begin{eqnarray}\label{bord2}
\big|\log\frac{|O_{h}(z)|}{h(\theta)}\big|\leq e^{41}\frac{\big|O_{h}(z)-O_{h}(\theta)\big|}{h(\theta)}.
\end{eqnarray}

Now since
$$ \big|\log\frac{h(\varphi)}{h(\theta)}\big|\leq \frac{\big|h(\varphi)-h(\theta)\big|}{h(\theta)},\qquad h(\varphi)\geq h(\theta),$$
Jensen's inequality gives
\begin{eqnarray*}\nonumber
&&{\Big(\frac{1}{2\pi}\int_{ h(\varphi)\geq h(\theta)}\frac{\big|\log\frac{h(\varphi)}{h(\theta)}\big|}{|e^{i\varphi}-z|^{2}} d\varphi \Big)^{2}
 =\Big(\frac{1}{2\pi (1-r^2)}\int_{h(\varphi)\geq h(\theta)}\frac{1-r^2}{|e^{i\varphi}-z|^{2}}{\big|\log\frac{h(\varphi)}{h(\theta)}\big|} d\varphi \Big)^{2}}
\\\label{bord1}\
&&\leq\frac{1}{2\pi h^{2}(\theta)(1-r^2)}\int_{h(\varphi)\geq h(\theta)}\frac{\big|h(\varphi)-h(\theta)\big|^{2}}{|e^{i\varphi}-z|^{2}} d\varphi ,
\quad z\in\D_{_h},
\end{eqnarray*}
which allows to control also the third integral in \eqref{bord0}:
\begin{eqnarray}\nonumber
\cI_{1}
&\lesssim&\int_{\D}\frac{\big|O_{h}(z)-O_{h}(\theta)\big|^{2}}{1-r}dA_{\alpha}(z)
+\int_{\D}\big(\int_{-\pi}^{\pi}\frac{\big|h(\varphi)-h(\theta)\big|^{2}}{|e^{i\varphi}-z|^{2}} d\varphi \big)dA_{\alpha}(z).
\end{eqnarray}
Again by Jensen's inequality,
\begin{eqnarray}\nonumber
|O_{h}(z)-O_{h}(\theta)|^2\leq\frac{1}{2\pi}\int_{-\pi}^{\pi}
\frac{1-r^2}{|e^{i\varphi}-z|^{2}}|O_h(\varphi)-O_h(\theta)|^2 d\varphi ,
\end{eqnarray}
for almost all points $z\in\D$ with respect to area Lebesgue measure.
This together with
\eqref{circ6} on the second term and then another application of 
\eqref{circ6} on the first term as well as Douglas' formula \eqref{DouglasFormulaalpha},
yield
\begin{eqnarray}\nonumber
\cI_{1}
&\lesssim&\int_{\D}\int_{-\pi}^{\pi}
\frac{\big|O_h(\varphi)-O_h(\theta)\big|^{2}}{|e^{i\varphi}-z|^{2}} d\varphi dA_{\alpha}(z)
+\int_{-\pi}^{\pi}\int_{-\pi}^{\pi}
\frac{\big|h(\varphi)-h(\theta)\big|^{2}}{|e^{i\varphi}-e^{i\theta}|^{2-\alpha}} d\varphi d\theta
\\\label{ss4}&\lesssim& \cD_\alpha(O_h).
\end{eqnarray}

Now we turn to the integral $\cI_{2}$. 
Again, we cannot use the triangular inequality directly in $\LL^2_h$ since we need to take
care of the sign of $\partial v_h/\partial \theta$. To this end,
we use $\LL^{2}_{h}=\D_{_h}\setminus\big\{\K_h\cup\LL^{1}_{h}\big\},$ 
where $\LL^{1}_{h}:=\LL^{1}_{h,\mu_{_h}}$. Then
\begin{eqnarray}\nonumber
\cI_{2}&\leq&\Big|\int_{\D}|O_{h}^{2}(z)| \frac{\partial v_{_h}}{\partial\theta}(z)dA_{\alpha}(z)\Big|
+\int_{\K_h} \big| O_{h}^{2}(z)\frac{\partial v_{_h}}{\partial\theta}(z)\big|dA_{\alpha}(z)
\\\label{bord3}&&+\int_{\LL^{1}_{h}}\big| O_{h}^{2}(z)\frac{\partial v_{_h}}{\partial\theta}(z)\big|dA_{\alpha}(z),
\end{eqnarray}
and so, by \eqref{motassa2} and Lemma \ref{harmonic},
\begin{eqnarray}\label{ss2}
\cI_{2}\lesssim\cD_{\alpha}(O_h)+
N_{\alpha}(h)+\int_{\LL^{1}_{h}}\big| O_{h}^{2}(z)\frac{\partial v_{_h}}{\partial\theta}(z)\big|dA_{\alpha}(z).
\end{eqnarray}
Since for $z\in \LL_h$ 
we have $|O_h(z)|\le 2h(\theta)$,
inequalities \eqref{circ5} and \eqref{cauchy} give
$|O_h^2\frac{\dst \partial v_h}{\dst \partial\theta}|\lesssim |\frac{\dst\partial O^2_h}{\dst\partial z}|
=2|O_h\frac{\dst \partial O_h}{\dst\partial z}|\le 2\times 2h(\theta)\times \frac{\dst 2h(\theta)}
{\dst 1-r}$,
and so
\begin{eqnarray}\nonumber
\int_{\LL^{1}_{h}}\big| O_{h}^{2}(z)\frac{\partial v_{_h}}{\partial\theta}(z)\big|dA_{\alpha}(z)
&\leq&8\alpha\int_{-\pi}^{\pi}h^{2}(\theta)\Big(\int_{r\leq 1-\mu_{_h}(\theta)}\frac{1}{(1-r)^{2-\alpha}}dr\Big)d\theta
\\\label{f2}&\leq&\frac{8\alpha}{1-\alpha}\int_{\mu_{_h}(\theta)<1}h^{2}(\theta)\mu^{\alpha-1}_{_h}(\theta)d\theta.
\end{eqnarray}
Combining inequalities \eqref{ss2} and \eqref{f2}, and applying Lemma \ref{dauglas},
\begin{eqnarray}\label{i2}
\cI_{2}\lesssim \frac{1}{1-\alpha}\cD_{\alpha}(O_h).
\end{eqnarray}
Hence, the desired result follows from the estimates \eqref{nsit1}, \eqref{ss11}, \eqref{ss4} and \eqref{i2}.
\end{proof}

\section{\bf The example} \label{sect6}

Recall that
for $0<\alpha<1$ and $\beta>0$,
we have defined the function
\begin{equation}\label{p0bis}
 h_{\beta}(\theta):=
\left\{
  \begin{array}{ll}
\displaystyle \frac{1}{\theta^{\frac{\alpha}{2}}\log^\beta\frac{\gamma}{\theta}}, & \qquad \theta\in (0,\pi] , \\
 \displaystyle c_0:=
h_{\beta}(\pi), 
& \qquad \theta
 \in (-\pi,0),
  \end{array}
\right.
\end{equation}
where the value of $\gamma$ ($=\pi e^{2\beta/\alpha}$) 
guarantees that $h_{\beta}$ is well defined,
decreasing on $(0,\pi]$.

We want to show the following result.

\begin{prop*}Let  $0<\alpha<1$ and $\beta>0$.
Then
\begin{enumerate}
  \item [(i)] For $N_\alpha(h_{\beta})<+\infty$ it is necessary and sufficient that $\beta>\frac{1}{2}.$
  \item [(ii)] For $O_{h_{\beta}}\in\cD_\alpha$ it is necessary and sufficient that $\beta>1-\frac{1}{2}\alpha.$
  \item [(iii)] For $\mathcal{C}_\alpha(h_{\beta})<+\infty$  it is necessary and sufficient that $\beta>1.$
\end{enumerate}
\end{prop*}
In order to not overload notation in our following discussions, we will set $h=h_{\beta}$.
Note that since $\alpha<1$, we can check easily that $h\in\mathcal{L}^{2}(\T)$ and satisfies the condition \eqref{logint}, and  hence $O_h\in\cH^{2}.$

For the convenience of the reader all estimates in the proof below will be done on
$[-\pi,\pi]$ rather than on $\T$.

\begin{proof}
\underline{Assertion (i)}

We have
\begin{eqnarray}\nonumber
N_\alpha(h)&\asymp&\int_{-\pi}^{\pi}\int_{-\pi}^{\pi}\frac{|h(\varphi)-h(\theta)|^2}{|e^{i\varphi}-e^{i\theta}|^{2-\alpha}}d\varphi d\theta
\\\nonumber
&=&
2\int_{0}^{\pi}\left[\int_{-\pi}^{0}\frac{\Big(\frac{1}{\theta^{\frac{\alpha}{2}}\log^\beta\frac{\gamma}{\theta}}
-c_0\Big)^2}{|e^{i\varphi}-e^{i\theta}|^{2-\alpha}}d\varphi\right] d\theta
\\\nonumber&&+
2\int_{0}^{\pi}\left[\int_{0<\theta<\varphi\leq \pi}\frac{\Big(
 \frac{1}{\theta^{\frac{\alpha}{2}}\log^\beta\frac{\gamma}{\theta}}
 -\frac{1}{\varphi^{\frac{\alpha}{2}}\log^\beta\frac{\gamma}{\varphi}}
\Big)^2}{|e^{i\varphi}-e^{i\theta}|^{2-\alpha}}d\varphi\right] d\theta
\\\label{p01}
&=:& \cI_1+\cI_2.
\end{eqnarray}
Since $|e^{i\varphi}-e^{i\theta}|=2|\sin\frac{\varphi-\theta}{2}|,$ then
\begin{eqnarray}\label{then}
|e^{i\varphi}-e^{i\theta}|
 \asymp \left\{
 \begin{array}{ll}
 |\varphi-\theta|, & \quad \text{ if } \theta\in(0,\pi/2) \text{ and }\varphi\in(-\pi,0),\\
 |\varphi-\theta+2\pi|, & \quad  \text{ if }\theta\in(\pi/2,\pi) \text{ and }\varphi\in(-\pi,0).
 \end{array}
 \right.
\end{eqnarray}
 We estimate the inner integral in the first term $\cI_1$:
\begin{eqnarray*}
 \int_{-\pi}^{0}\frac{\Big(\frac{1}{\theta^{\frac{\alpha}{2}}\log^\beta\frac{\gamma}{\theta}}
-c_0\Big)^2}{|e^{i\varphi}-e^{i\theta}|^{2-\alpha}}d\varphi
 \simeq \Big(\frac{1}{\theta^{\frac{\alpha}{2}}\log^\beta\frac{\gamma}{\theta}}
-c_0\Big)^2 \frac{1}{\theta^{1-\alpha}},\qquad \theta\in (0,\pi/2).
\end{eqnarray*}
Now taking the outer integral in $\cI_1$, the convergence of which does not depend
on the behavior on $(\pi/2,\pi)$,
we get
\[
 \cI_1<\infty \quad \Leftrightarrow\quad
 \int_0^{\pi/2}\Big(\frac{1}{\theta^{\frac{\alpha}{2}}\log^\beta\frac{\gamma}{\theta}}
-c_0\Big)^2 \frac{1}{\theta^{1-\alpha}}d\theta<\infty,
\]
and since on $(0,\pi/2)$ we have
\[
\frac{1}{\theta^{\frac{\alpha}{2}}\log^\beta\frac{\gamma}{\theta}}
-c_0\asymp \frac{1}{\theta^{\frac{\alpha}{2}}\log^\beta\frac{\gamma}{\theta}},
\]
this yields
\[
 \cI_1<\infty\quad \Leftrightarrow \quad
 \int_0^{\pi/2}\frac{1}{\theta^{{\alpha}}\log^{2\beta}\frac{\gamma}{\theta}}
  \frac{1}{\theta^{1-\alpha}}d\theta=
 \int_0^{\pi/2}\frac{1}{\theta\log^{2\beta}\frac{\gamma}{\theta}}<\infty.
\]
Hence
\[
 \cI_1<\infty \quad \Leftrightarrow \quad 2\beta>1,
\]
which yields the necessity in (i) of the proposition.
\\

Let us discuss the second integral $\cI_2$. Since $\frac{\dst h(\theta)-h(\varphi)}{\dst
\theta-\varphi}=h'(\xi)$ for some $\theta\le\xi\le \varphi$,
we will be interested in the derivative of $h$:
\begin{equation}\label{cajol}
|h'(\theta)|\asymp \frac{1}{\theta^{1+\frac{1}{2}\alpha}\log^\beta\frac{\gamma}{\theta}},\qquad  \theta\in]0,\pi[.
\end{equation}
As above
\begin{eqnarray}\label{also}
|e^{i\varphi}-e^{i\theta}| \asymp |\varphi-\theta|, \qquad \theta,\varphi\in(0,\pi).
\end{eqnarray}
So
\begin{eqnarray}\nonumber
\cI_2
&=&2\int_{0}^{\pi}\int_{{\theta}\leq\varphi\leq 2\theta\atop \varphi\le \pi}
+2\int_{0}^{\pi}\int_{2\theta\leq\varphi\leq \pi}
\\\nonumber
&\lesssim&2\int_{0}^{\pi}\Big(\big(\sup_{\theta\leq\varphi\leq 2\theta\atop \varphi\leq \pi} |h'(\varphi)|\big)^2
\int_{\theta\leq\varphi\leq 2\theta\atop \varphi\le\pi}(\varphi-\theta)^{\alpha}d\varphi\Big)d\theta
\\\nonumber
&&+ 2\int_{0}^{\pi}
\frac{1}{\theta^{\alpha}\log^{2\beta}\frac{\gamma}{\theta}}
\Big(\int_{2\theta\le\varphi\le \pi}
 \frac{1}{(\varphi-\theta)^{2-\alpha}}d\varphi\Big)d\theta
\\\label{p03}
&\lesssim&
 \int_{0}^{\pi} \frac{1}{\theta\log^{2\beta}\frac{\gamma}{\theta}}d\theta,
\end{eqnarray}
which converges when $2\beta>1$. As a result we deduce the sufficient
part in (i) of the proposition.\\

\underline{Assertion (ii)}

Now we set
$$\lambda(\theta):=\frac{|\theta|}{4},\qquad |\theta|\leq \pi.$$
It is clear that $\lambda\in\Lambda$ and
\begin{equation}\label{p111}
|{\theta}|\asymp|\varphi|,\qquad
|\theta|\leq \pi \text{ and }|e^{i\varphi}-e^{i\theta}|\leq\lambda(\theta).
\end{equation}
From \eqref{p111} and the explicit form of $h$, we deduce that
\begin{equation}\label{pequiv}
h(\varphi){\asymp}h(\theta) ,\qquad 0<\theta,\varphi  \leq \pi \text{ and }|e^{i\varphi}-e^{i\theta}|\leq\lambda(\theta).
\end{equation}
Hence
\begin{equation}\label{p12}
\big|\log \frac{h(\theta)}{h(\varphi)}\big|\asymp \frac{|h(\theta)-h(\varphi)|}{h(\theta)},\qquad0<\theta,\varphi \leq \pi \text{ and }|e^{i\varphi}-e^{i\theta}|\leq\lambda(\theta).
\end{equation}

It is also obvious that when $-\pi<\theta<0$, then for no $\varphi$ we can have
$h(\varphi)\le \frac{1}{2}h(\theta)$, so that in the integration for $\tilde{n}_{\alpha}$ we
only need to integrate for $\theta\in (0,\pi)$.

Thus
\begin{eqnarray}\nonumber
\widetilde{n}_\alpha(h,\lambda)
&=&\int_{0}^{\pi}h^{2}(\theta) \int_{h(\varphi)\le\frac{1}{2}h( \theta)\atop |e^{i\varphi}-e^{i\theta}|\leq \lambda(\theta)}\frac{\log \frac{h(\theta)}{h(\varphi)}}{|e^{i\varphi}-e^{i\theta}|^{2-\alpha}}d\varphi d\theta
\\\nonumber
&\lesssim&\int_{-\pi}^{\pi} \int_{|e^{i\varphi}-e^{i\theta}|\leq \lambda(\theta)}\frac{\big|h(\theta)- h(\varphi)\big|^2}{|e^{i\varphi}-e^{i\theta}|^{2-\alpha}}d\varphi d\theta
\\\label{p13}&\lesssim& N_\alpha(h),
\end{eqnarray}
which, by assertion (i), converges when $\beta>1/2$, and so also when $\beta>1-\alpha/2$.

It remains to estimate $n_\alpha(h,\lambda).$
For the same reason as above, when computing $n_{\alpha}$ we only need to integrate
over $(0,\pi)$:
\begin{eqnarray}\nonumber
n_{\alpha}(h,\lambda)
&=&\int_{0}^{\pi}h^{2}(\theta) \Big(\int_{h(\varphi)\le \frac{1}{2}h(\theta)\atop |e^{i\varphi}-e^{i\theta}|\geq \lambda(\theta)}\frac{\log \frac{h(\theta)}{h(\varphi)}}{|e^{i\varphi}-e^{i\theta}|^2}d\varphi\Big)^{1-\alpha}d\theta
\\\nonumber&\lesssim& \int_{0}^{\pi}h^{2}(\theta) \Big(\int_{|e^{i\varphi}-e^{i\theta}|\geq \lambda(\theta)}\frac{\log \frac{h(\theta)}{c_0}}{|e^{i\varphi}-e^{i\theta}|^2}d\varphi\Big)^{1-\alpha}d\theta.
\\\nonumber&\lesssim& \int_{0}^{\pi}h^{2}(\theta) \Big(\frac{\log \frac{h(\theta)}{c_0}}{\theta}\Big)^{1-\alpha}d\theta.
\end{eqnarray}
We have
\begin{equation}\label{al}
 \log\frac{h(\theta)}{c_0}\asymp \frac{\alpha}{2}\log\frac{\gamma}{\theta},\qquad \theta\in (0,\pi).
\end{equation}
Hence
\[
 h^{2}(\theta) \Big(\frac{\log \frac{h(\theta)}{c_0}}{\theta}\Big)^{1-\alpha}\asymp \frac{1}{\theta\log^{2\beta-1+\alpha}\frac{\gamma}{\theta}},\qquad \theta\in (0,\pi).
\]
We get
\beqa
\int_{0}^{\pi}h^{2}(\theta) \Big(\frac{\log \frac{h(\theta)}{c_0}}{\theta}\Big)^{1-\alpha}d\theta
 \asymp \int_{0}^{\pi} \frac{1}{\theta\log^{2\beta-1+\alpha}\frac{2\pi}{\theta}}d\theta,
\eeqa
which converges when $2\beta-1+\alpha>1$ or $\beta>1-\frac{\alpha}{2}$.
This achieves the sufficiency in (ii).
\\

Let us turn to the necessity of this condition.
We fix a point  $\theta\in]0,{\pi}[.$ Observe that 
\begin{eqnarray}\nonumber
\theta a_{_{h,\theta}}(\theta)
 &=& \frac{\theta}{2\pi}\int_{h(\varphi)\le \frac{1}{2}h(\theta)\atop |e^{i\varphi}-e^{i\theta}|\geq \theta}\frac{\log \frac{h(\theta)}{h(\varphi)}}{|e^{i\varphi}-e^{i\theta}|^2}d\varphi
\\\nonumber&\geq& \frac{\theta}{2\pi}\int_{\varphi\in(-\pi,0)\atop |e^{i\varphi}-e^{i\theta}|\geq \theta}\frac{\log \frac{h(\theta)}{c_0}}{|e^{i\varphi}-e^{i\theta}|^2}d\varphi
\\\label{q2}&\gtrsim&
\log h(\theta) \longrightarrow \infty,\qquad \text{ as }\theta\to 0.
\end{eqnarray}
In particular, there is a number $0<\delta\leq \pi/4$, 
such that
\begin{eqnarray}\label{q22}
 \theta a_{_{h,\theta}}(\theta)>2,\qquad 0<\theta<\delta.
\end{eqnarray}
It follows
$$\mu_{_h}(\theta)\leq \theta,\qquad 0<\theta<\delta.$$
Thus, for $0<\theta<\delta,$
\begin{eqnarray}\nonumber
a_{_{h}}(\theta)&=& \frac{1}{2\pi}\int_{h(\varphi)\le\frac{1}{2}h(\theta)\atop |e^{i\varphi}-e^{i\theta}|\geq \mu_{_h}(\theta)}\frac{\log \frac{h(\theta)}{h(\varphi)}}{|e^{i\varphi}-e^{i\theta}|^2}d\varphi
\\\nonumber
&\geq& \frac{1}{2\pi}\int_{\varphi\in(-\pi,0)\atop |e^{i\varphi}-e^{i\theta}|\geq \theta}\frac{\log \frac{h(\theta)}{c_0}}{|e^{i\varphi}-e^{i\theta}|^2}d\varphi
\\\label{q3}
&\gtrsim& \frac{\log h(\theta)}{\theta},
\end{eqnarray}
which gives
\begin{eqnarray}\nonumber
n_{\alpha}(h)&=&\int_{\mu_{_h}(\theta)<1}h^{2}(\theta) a^{1-\alpha}_{_{h}}(\theta)d\theta
\geq\int_{0}^{\delta}h^{2}(\theta) a^{1-\alpha}_{_{h}}(\theta)d\theta
\\\nonumber
&\gtrsim&\int_0^{\delta}\frac{1}{\theta^{\alpha}\log^{2\beta}\frac{\gamma}{\theta}}
 \left(\frac{\log\frac{\gamma}{\theta}}{\theta}\right)^{1-\alpha}d\theta
 \\\label{q4}&\asymp&\int_{0}^{\delta} \frac{1}{{\theta}\log^{2\beta+\alpha-1}\frac{\gamma}{{\theta}}}d\theta.
\end{eqnarray}
Hence, the condition $\beta>1-\frac{1}{2}\alpha$ is necessary for $n_{\alpha}(h)<+\infty$,
which finishes the proof of the second assertion.
\\

\underline{Assertion (iii)}

Clearly, there is a constant $k>1$ such that
$h(\theta)>k h(\varphi)=c_0$ when
$\varphi\in [-\pi,0[$  and $\theta\in[0,\pi/2[$. This yields
\begin{eqnarray}\label{equivp31}
h^2(\theta)-h^2(\varphi)\asymp h^2(\theta),\qquad \varphi\in [-\pi,0[\text{ and }\theta\in[0,\pi/2[.
\end{eqnarray}
Using \eqref{then} and \eqref{al},
\begin{eqnarray}\nonumber
\mathcal{C}_\alpha(h)&\gtrsim&
\int_0^{{\pi/2}}\Big(\int_{-\pi}^{0}
\frac{\big(h^2(\theta)-h^2(\varphi)\big)\log \frac{h(\theta)}{h(\varphi)}}{|{\varphi}-{\theta}|^{2-\alpha}} d\varphi\Big)d\theta
\gtrsim \int_0^{{\pi/2}}\frac{h^2(\theta)\log \frac{h(\theta)}{c_0}}{{\theta}^{1-\alpha}}
d\theta\\ \nonumber
\label{p11}&\asymp&\int_{0}^{\pi/2} \frac{1}{{\theta}\log^{2\beta-1}\frac{\gamma}{{\theta}}}d\theta.
\end{eqnarray}
Hence the condition $\beta>1$ is necessary for $\mathcal{C}_\alpha(h)<+\infty.$

We now show the sufficiency of this condition.
Since the function $h$ is constant on $(-\pi,0)$, there is nothing to prove
when $\varphi,\theta\in (-\pi,0)$. We now consider the case when $\theta\in (0,\pi)$ and $\varphi\in (-\pi,0).$
We have in view of \eqref{equivp31} and \eqref{then}
\begin{eqnarray}\nonumber
\lefteqn{\cI:=\int_0^{\pi}\Big(\int_{-\pi}^{0}
\frac{\big(h^2(\theta)-h^2(\varphi)\big)\log \frac{h(\theta)}{h(\varphi)}}{|e^{i\varphi}-e^{i\theta}|^{2-\alpha}} d\varphi\Big)d\theta}
\\\nonumber&&\asymp\int_0^{\pi/2}\Big(\int_{-\pi}^{0}
\frac{h^2(\theta)\log \frac{h(\theta)}{c_0}}{|\varphi-\theta|^{2-\alpha}} d\varphi\Big)d\theta
+\int_{\pi/2}^{\pi}\Big(\int_{-\pi}^{0}
\frac{h^2(\theta)\log \frac{h(\theta)}{c_0}}{|e^{i\varphi}-e^{i\theta}|^{2-\alpha}} d\varphi\Big)d\theta
\\\nonumber&&\lesssim \int_{0}^{\pi/2} \frac{1}{{\theta}\log^{2\beta-1}\frac{\gamma}{{\theta}}}d\theta
+\int_{\pi/2}^{\pi} h^2(\theta)\log \frac{h(\theta)}{c_0}
 \int_{-\pi}^0 \frac{1}{|\varphi-\theta+2\pi|^{2-\alpha}}d\varphi d\theta.
\end{eqnarray}
The second term is of no harm since $h^2(\theta)\log \frac{\dst h(\theta)}{\dst c_0}$ is
bounded on $[\pi/2,\pi]$. Hence $\cI$ converges if and only if
$\int_{0}^{\pi/2} \frac{1}{{\theta}\log^{2\beta-1}\frac{\gamma}{{\theta}}}d\theta$ converges,
which happens when $\beta>1.$
It remains to check the case when $\varphi,
\theta\in (0,\pi)$. By  \eqref{also},
\begin{eqnarray*}\nonumber
\lefteqn{\int_0^{\pi}\int_{0}^{\pi}
\frac{\big(h^2(\theta)-h^2(\varphi)\big)\log \frac{h(\theta)}{h(\varphi)}}{|e^{i\varphi}-e^{i\theta}|^{2-\alpha}} d\varphi d\theta}
\\&&=
2 \int_0^{\pi}\Big(\int_{0<\varphi<\pi}
\frac{\big(h^2(\theta)-h^2(\varphi)\big)\log \frac{h(\theta)}{h(\varphi)}}{|e^{i\varphi}-e^{i\theta}|^{2-\alpha}} d\varphi\Big)d\theta
\\\nonumber&&\asymp
\int_0^{\pi}\Big(\int_{0<\varphi<\pi}
\frac{\big(h^2(\theta)-h^2(\varphi)\big)\log \frac{h(\theta)}{h(\varphi)}}{|\varphi -\theta|^{2-\alpha}} d\varphi\Big)d\theta.
\end{eqnarray*}
Clearly
\begin{eqnarray}\nonumber
&& \int_0^{{\pi}}\Big(\int_{0<\varphi<\pi\atop |\varphi-\theta|\geq\frac{1}{2}\theta}
\frac{\big(h^2(\theta)-h^2(\varphi)\big)\log \frac{h(\theta)}{h(\varphi)}}{|{\varphi}-{\theta}|^{2-\alpha}} d\varphi\Big)d\theta
 \\&\le&  \int_0^{{\pi}}\Big(\int_{|\varphi-\theta|\geq\frac{1}{2}\theta}
\frac{h^2(\theta)\log \frac{h(\theta)}{c_0}} {|{\varphi}-{\theta}|^{2-\alpha}} d\varphi\Big)d\theta
\lesssim \int_{0}^{\pi} \frac{1}{{\theta}\log^{2\beta-1}\frac{2\pi}{{\theta}}}d\theta,
\end{eqnarray}
which as in the previous estimate is bounded when $\beta>1$.
Finally we consider the integral for $|\varphi-\theta|\leq\frac{1}{2}\theta.$ We observe first that
in this case, as already discussed earlier,
$h(\theta)\asymp h(\varphi)$ and
\[
 \left|\log\frac{h(\theta)}{h(\varphi)}\right|
 \lesssim \frac{|h(\theta)-h(\varphi)|}{h(\theta)}.
\]
Hence
\begin{eqnarray*}
 &&\int_0^{{\pi}}\Big(\int_{0<\varphi<\pi\atop h(\varphi)\leq h(\theta)\text{ and }|\varphi-\theta|\leq\frac{1}{2}\theta}
\frac{\big(h^2(\theta)-h^2(\varphi)\big)\log \frac{h(\theta)}{h(\varphi)}}{|{\varphi}-{\theta}|^{2-\alpha}} d\varphi\Big)d\theta
 \\&\lesssim&
\int_0^{{\pi}}\Big(\int_{0}^{\pi}
\frac{\big|h(\theta)-h(\varphi)\big|^2 }{|{\varphi}-{\theta}|^{2-\alpha}} d\varphi\Big)d\theta
\lesssim N_{\alpha}(h),
\end{eqnarray*}
which converges when $2\beta>1$ and in particular when $\beta>1.$
\end{proof}

\noindent{\bf Acknowledgements.} The second named author would like to thank the
referee for careful reading of the manuscript.

\end{document}